\documentclass[reqno,a4paper]{amsart}

\usepackage[utf8]{inputenc}
\usepackage[T1]{fontenc}
\usepackage[pagebackref]{hyperref}
\usepackage{amsfonts}
\usepackage{bbm}
\usepackage{mathtools}
\usepackage{enumerate}
\usepackage{ulem}
\usepackage[left=2.5cm,right=2.5cm]{geometry}
\allowdisplaybreaks
\newtheorem{theorem}{Theorem}[section]
\newtheorem{lemma}[theorem]{Lemma}
\newtheorem{cor}[theorem]{Corollary} 
\newtheorem{prop}[theorem]{Proposition} 

\theoremstyle{definition}

\theoremstyle{remark}
\newtheorem{remark}[theorem]{Remark}

\numberwithin{equation}{section}

\def\ii{{\bf i}}
\def\jj{{\bf j}}
\def\kk{{\bf k}}

\newcommand{\ka}{\mathfrak{c}}
\renewcommand{\i}{\mathfrak{i}}
\newcommand{\id}{\mathrm{Id}}
\newcommand{\ind}[1]{\mathbf{1}_{\{#1\}}}
\newcommand{\E}{\mathbb{E}} 
\newcommand{\p}{\mathbb{P}} 
\newcommand{\R}{\mathbb{R}} 

\newcommand{\N}{\mathbb{N}}
\newcommand{\om}{o_m}
\newcommand{\Om}{\mathcal{O}_m}
\newcommand{\Wass}{\mathcal{W}}
\newcommand{\rng}{[n]^{\underline{d}}} 
\newcommand{\Rng}[1]{[n]^{\underline{#1}}}
\DeclareMathOperator{\Var}{Var}	
\DeclareMathOperator{\Cov}{Cov}	
\DeclareMathOperator{\spa}{span}
\DeclareMathOperator{\Vol}{Vol}
\DeclareMathOperator{\Ker}{Ker}

\title[Limit theorems for the volumes of small codimensional random
  sections of $\ell_p^n$-balls]{Limit theorems for the volumes of
   small codimensional random sections of $\ell_p^n$-balls}

\author{Rados{\l}aw Adamczak}
\address{Institute of Mathematics, University of Warsaw, ul. Banacha 2, 02-097 Warszawa, Poland}
\email{r.adamczak@mimuw.edu.pl}
\author{Peter Pivovarov}
\address{Mathematics Department, University of Missouri, Columbia, Missouri 65211}
\email{pivovarovp@missouri.edu}
\author{Paul Simanjuntak}
\address{Mathematics Department, University of Missouri, Columbia, Missouri 65211}
\email{simanjuntakp@missouri.edu}

\date{}
\thanks{Radosław Adamczak was supported by the  National Science Center, Poland via the Sonata Bis  grant  no.\ 2015/18/E/ST1/00214. Peter Pivovarov was supported by NSF-Grant DMS-2105468 and Simons Foundation grant \#635531}

\begin{document}
\maketitle

\begin{abstract}
  We establish Central Limit Theorems for the volumes of intersections
  of $B_{p}^n$ (the unit ball of $\ell_p^n$) with uniform random subspaces of codimension $d$
  for fixed $d$ and $n\to \infty$. As a corollary we obtain higher
  order approximations for expected volumes, refining previous results
  by Koldobsky and Lifschitz and approximations obtained from the
  Eldan--Klartag version of CLT for convex bodies. We also obtain a
  Central Limit Theorem for the Minkowski functional of the
  intersection body of $B_p^n$, evaluated on a random vector
  distributed uniformly on the unit sphere.
\end{abstract}

\section{Introduction}

An important aspect of stochastic geometry is the investigation of
volumes of random sets. They have been studied in a variety of
contexts, including for instance volumes of convex hulls of
i.i.d. Gaussian vectors in $\R^m$ (e.g., B\'{a}r\'{a}ny and Vu
\cite{MR2330981}, Calka and Yukich \cite{MR3405618}), points selected
from the boundary of a fixed convex set (e.g., Sch\"{u}tt and Werner
\cite{MR2083401}, Reitzner \cite{MR1885651}, Th\"{a}le
\cite{MR3802312}, Vu \cite{MR2221249}), projections of
high-dimensional cubes onto a random subspace of fixed dimension
(Paouris--Pivovarov--Zinn \cite{MR3230006},
Kabluchko-Prochno-Th\"{a}le \cite{MR4260512}). In the high dimensional
setting, asymptotics and phase transitions for the expected volume of
i.i.d. random points selected from the vertices of the unit cube or
more general product distributions were investigated by Dyer--F\"{u}redi--McDiarmid \cite{MR1139489}, and Gatzouras and Giannopoulos
\cite{MR2460902}. Recently, the case of points drawn from a simplex
and from non-product convex measures has been studied by
Frieze--Pegden--Tkocz \cite{MR4157095} and
Chakraborti--Tkocz--Vritsiou \cite{MR4278336}.  Stochastic versions of
isoperimetric inequalities rely on images of deterministic convex sets
by random mappings (Paouris and Pivovarov
\cite{MR2921184,MR3038532}, Cordero-Erausquin et
al. \cite{MR3368101}). Another important line of research is devoted
to sets obtained from point processes defined on spaces of geometric
objects, e.g., random tesselations (Gusakova and Th\"{a}le
\cite{MR4213157}), random cylinder processes (Beci et
al. \cite{MR4169169}). Let us finally mention work concerning unions
and Minkowski sums of random sets (see the monograph by Molchanov
\cite{MR3751326}).  While far from exhaustive, the lines of research
above give an indication of the rich and diverse perspectives on
volumetric questions in stochastic geometry.

Most of the early results focused on first order asymptotics in the
sense of expectation or convergence (almost sure or in
probability). More recent developments also treat concentration
inequalities, small ball probabilities, large deviations or weak limit
theorems. In particular central limit theorems for the volume or
log-volume in the respective models were established in
\cite{MR2330981,MR3802312,
  MR3230006,MR4172611,MR4213157,gusakova2022volume}. The three former
references treat convex polytopes in a fixed dimension, whereas the
other ones random simplices for the dimension tending to infinity.

In this article we focus on another model of random sets, namely on
sections of high dimensional origin-symmetric bodies by random
subspaces (i.e., subspaces drawn from the Haar measure on the
corresponding Grassmann manifold).  When the bodies in question are
convex, this model has played a central role in geometric functional
analysis, especially via probabilistic methods put forth by Vitali
Milman in his proof of Dvoretzky's theorem \cite{MR0293374}. Over the
years, this has grown into the whole new area of Asymptotic Geometric
Analysis (see the classical book \cite{MR856576} and recent monographs
\cite{MR3331351,MR3185453}).  Convexity is often used by invoking
duality, especially between sections and projections, but some key
results actually extend to star-shaped sets (e.g., Litvak et
al. \cite{MR1469422, MR1645952}). Moreover, star-shaped sets also
furnish deeper dualites in convex geometry, especially in dual
Brunn-Minkowski theory (e.g., Lutwak et al. \cite{MR380631, MR963487,
  MR1890647, MR3573332}).  However, investigation of the asymptotic
distributions of the volumes of random sections of star bodies from a
stochastic geometry perspective is a less explored path. Our focus
here is on asymptotic properties of random sections of $\ell^n_p$-balls, including the star-shaped case when $p\in (0,1)$.

Geometric properties of random sections depend strongly on the
relation between the dimensions of the ambient space and the
section. For convex bodies in special positions, low dimensional
sections are generically spherical and their approximate radia can be
calculated in terms of certain geometric characteristics of the body
(as stated in Milman's version of Dvoretzky's theorem). Accurate
volume approximation in this regime is just one of the important
consequences.  At the other extreme, one has sections of small
codimension. The asymptotic behaviour of their volumes has been
obtained more recently and is directly related to the celebrated
Klartag's CLT for convex bodies \cite{MR2285748,MR2311626} and
subsequent results by Eldan and Klartag \cite{MR2402109}, which we
will now recall briefly (and somewhat informally).  We will also
restrict attention to convex bodies, even though these results hold in
the more general setting of log-concave measures. Below by $|\cdot|$
we denote the standard Euclidean norm on $\R^n$.

Assume thus that $K$ is a convex body in $\R^n$ in isotropic
position, e.g., a random vector $X$ distributed uniformly in $K$ has
mean zero and covariance matrix equal to the identity. Let $E$ be a
random $k$-dimensional subspace of $\R^n$ distributed according to the
Haar measure on the Grassmanian $G_{n,k}$. The Central Limit Theorem
due to Klartag asserts that there exists a universal constant $c>0$
such that with probability at least $1 - e^{-cn^c}$ on the
Grassmanian, the total variation between $\gamma_E$ -- the standard
Gaussian measure on $E$ (i.e., the measure with density $g_E(x) =
(2\pi)^{-k/2} \exp(-|x|^2/2)$ with respect to the $k$ dimensional
Lebesgue measure on $E$) and the law $\mathcal{L}(P_E X)$, where $P_E$
is the orthogonal projection onto $E$, satisfies
\begin{displaymath}
  \|\gamma_E - \mathcal{L}(P_E X)\|_{TV} \le \frac{1}{cn^c}.
\end{displaymath}
Thus, informally, for $k \le n^c$ and $n$ tending to infinity, almost all $k$-dimensional marginals of $X$ are almost Gaussian.

The total variation estimate given by Klartag was subsequently
complemented by Eldan and Klartag \cite{MR2402109} with pointwise
approximation of density. It turns out that with probability $1 -
e^{-cn^c}$, where $c$ is again a positive universal constant, for
$x\in E$ with $|x|<n^c$, the density $f$ of $P_E X$ satisfies
\begin{displaymath}
  \Big|\frac{f(x)}{g_E(x)} - 1\Big| \le \frac{1}{cn^c}.
\end{displaymath}
This result is of particular importance from the point of view of
volumes, since $f(x)$ equals to the ratio ${\Vol_{n-k}(K\cap
  (x+E^\perp))}/{\Vol_{n}(K)}$, where $\Vol_i$ stands for the
$i$-dimensional Lebesgue measure and $E^\perp$ is the orthogonal
complement of $E$.  The above approximation holds for isotropic convex
bodies. By an appropriate scaling one obtains that if $K_n \subseteq
\R^n$ is a sequence of convex bodies such that the random vectors
uniformly distributed on $K_n$ satisfy $\E X_n = 0$, $\Cov(X_n) = c_n
\id$, and $H_n$ is a random subspace of $K_n$ of \emph{codimension} $d
\le n^c$, then with probability tending to one as $n\to \infty$
\begin{align}\label{eq:1st-order-approximation}
\Vol(K_n \cap H_n) = \frac{1}{(2\pi c_n)^{d/2}} \Vol(K_n)(1+o(1)).
\end{align}
The result by Eldan and Klartag provides us thus in particular with the \emph{first order} approximations for the volumes of random section of \emph{fixed codimension} $d$ of high dimensional convex bodies.

The goal of this article is to complement \eqref{eq:1st-order-approximation} with Central Limit Theorems in the special case of $\ell_p^n$-balls.
Recall that for $p \in (0,\infty)$, the ball $B_p^n$ (the unit ball in the space $\ell_p^n$) is defined as
\begin{displaymath}
  B_p^n = \{x = (x_1,\ldots,x_n) \in \R^n \colon \sum_{i=1}^n |x_i|^p \le 1\}.
\end{displaymath}
whereas
\begin{displaymath}
  B_p^\infty = \{x\in \R^n\colon \max_{i\le n} |x_i|\le 1\}.
\end{displaymath}
$B_p^n$ is a symmetric star-shaped set and for $p \ge 1$ it is convex. One can verify that
\begin{align}\label{eq:volume-Bpn}
  \Vol_n(B_p^n) = \frac{(2\Gamma(1+1/p))^n}{\Gamma(1+ n/p)},
\end{align}
which via Stirling's approximation shows in particular that for fixed $p$ and large $n$, $\Vol_n(B_p^n)^{1/n}$ behaves like $c(p)/n^{1/p}$ for some constant $c(p)$.

Finite-dimensional $\ell_p$-balls are of interest from various points
of view, including the functional analytic, geometric, and
probabilistic. The interplay between these viewpoints is apparent in
early research on sections of $\ell_p^n$-balls (e.g., Ball
\cite{MR1008726}, Meyer and Pajor \cite{MR960226}, Koldobsky
\cite{MR1656857}). In fact, questions about low codimensional
sections of convex bodies have been a major driving force in convex
geometry in the last thirty years \cite{MR2132704, MR3185453}. From a
probabilistic perspective, uniform distributions on $B_p^n$ provide
examples of natural non-product distributions that nevertheless
exhibit properties that arise in the classical theory of independent
random variables. In recent years many strong results in this
direction have been established. In particular, Schechtman--Zinn
\cite{MR1796723}, Sodin \cite{MR2446328} and Lata{\l}a--Wojtaszczyk
\cite{MR2449135} studied concentration and isoperimetric properties.
Alonso-Guti\'{e}rrez--Prochno--Th\"{a}le \cite{MR4003577,MR3806754},
Gantert--Kim--Ramanan \cite{MR3737915} and Kabluchko-Prochno-Th\"{a}le
\cite{MR3904638,MR4216415} investigated limit theorems together with
large and moderate deviations for various norms of vectors drawn at
random from $B_p^n$, as well as their random projections. Naor and
Romik \cite{MR1962135,MR2262841} studied proximity of the normalized
cone and surface measures on the boundary of $B_p^n$ and its
consequences for concentration. Eskenazis--Nayar--Tkocz
\cite{MR3846841,MR3828740} established optimal constants in Khintchine
inequalities for vectors distributed uniformly on $B_p^n$. Another
reason for investigating probabilistic aspects of $B_p^n$ is
that probabilistic tools can be used in the study of other geometric
aspects, seemingly purely deterministic. This approach has been
initiated in the work by Schechtman and Zinn \cite{MR1015684} and
further continued in the seminal paper \cite{MR2123199} by Barthe et
al., and more recently, e.g., by Chasapis, Eskenazis, Nayar, Tkocz
(see
\cite{MR3846841,MR3828740,MR4055953,chasapis2021slicing}). Geometric
results obtained this way include among others monotonicity of various
geometric quantities and identification of subspaces of
maximal/minimal volume.  Let us mention that some of the results
mentioned above are special cases of statements conjectured for
general log-concave measures. In fact, also the CLT for convex bodies
was first established by Anttila-Ball-Perissinaki \cite{MR1997580} for
$B_p^n$ rescaled to the isotropic position.

The first order asymptotics for the expected volume of $B_p^n\cap
H_n$, where $H_n$ is a random subspace of fixed codimension $d$,
asymptotically equivalent to \eqref{eq:1st-order-approximation} (which
formally covers the case of $p \ge 1$), were obtained before the
result by Eldan--Klartag by Koldobsky and Lifschitz \cite{MR1796717}
through a Fourier-analytic approach (in fact the analysis in
\cite{MR1796717} covers not only the convex case but the full range of
$p>0$). These authors obtained asymptotics also in the cases of
sections of proportional and fixed \emph{dimension}. The equation
\eqref{eq:1st-order-approximation} specialized to $B_p^n$, after some
calculations concerning the covariance matrix gives (see
\cite[p. 490]{MR2123199} for an explicit formula for moments of
coordinates) that with probability tending to one as $n\to \infty$,
\begin{align}\label{eq:1st-order-B_pn}
\Vol_{n-d}(B_p^n\cap H_n) = \Big(\frac{3\Gamma\Big(1+\frac{1}{p}\Big)\Gamma\Big(1 + \frac{n+2}{p}\Big)}{2\pi \Gamma\Big(1+\frac{3}{p}\Big)\Gamma\Big(1+\frac{n}{p}\Big)}\Big)^{d/2} \Vol_n(B_p^n) (1+ o(1)).
\end{align}
(with $\Vol_n(B_p^n)$ given by \eqref{eq:volume-Bpn}). The results by Koldobsky--Lifschitz \cite{MR1796717} give the same asymptotic behaviour also for the expected volume. We remark that different approaches to the first order asymptotics yield different explicit expressions, which turn out to be equivalent thanks to Stirling's formula. In particular it can be read from the above asymptotics that

\begin{align}\label{eq:1st-order-B_pn-normalized}
  \frac{\Vol_{n-d}(B_p^n\cap H_n)}{\Vol_{n-d}(B_p^{n-d})} \stackrel{n\to \infty}{\to} a_{p,d} := \Big(\frac{\Gamma(1/p)}{\Gamma(3/p)}\Big)^{d/2}\frac{2^{d/2}\Gamma(1+1/p)^d }{\pi^{d/2}}
\end{align}
in probability. This formulation is more convenient for higher order analysis than \eqref{eq:1st-order-B_pn} as it allows to absorb certain normalization factors.

Our main results allow to complement this approximation with Central Limit Theorems. We postpone the precise formulation, which involves additional quite complicated formulas, to Section \ref{sec:main-results} and here we just state them in a simplified form:

\begin{itemize}
\item For $p \in (0,2)$ and any fixed $d > 0$ we show that for some explicit constants $b_{p,d}, \Sigma_{p,d}^2$, the random variable
\begin{displaymath}
  n^{3/2}\Big(\frac{\Vol_{n-d}(H_n\cap B_p^n)}{\Vol_{n-d}(B_p^{n-d})} -   a_{p,d} - \frac{1}{n}b_{p,d}\Big)
\end{displaymath}
converges weakly and in all Wasserstein distances $\Wass_{q}$ for $q > 0$, to a Gaussian random variable with mean zero and variance $\Sigma_{p,d}^2$ (recall that $a_{p,d}$ is defined in \eqref{eq:1st-order-B_pn-normalized}). This is the content of Theorem \ref{thm:CLT}.

\item As a corollary to the CLT in Wasserstein distance we obtain higher order approximations of $\E \Vol_{n-d}(B_p^n\cap H_n)$ for $p \in (0,2)$ of the form
\begin{displaymath}
   \E \Vol_{n-d}(B_p^n\cap H_n) = \frac{\Big(2\Gamma\Big(1+\frac{1}{p}\Big)\Big)^{n-d}}{\Gamma\Big(1+\frac{n-d}{p}\Big)}\Big(a_{p,d} + \frac{b_{p,d}}{n} + o(n^{-3/2})\Big)
\end{displaymath}
(Corollary \ref{cor:mean-approximation}).

\item For $d=1$ we extend the weak convergence to arbitrary $p > 0$ (Theorem \ref{thm:d=1-p-arbitrary}).

\item For $p = \infty$ and $d=1$ we establish similar weak convergence also for non-central sections parallel to $H$ (Theorem \ref{thm:cube}).

\item Reinterpreting the result for $d=1$ as a CLT for the radial
  function of the \emph{intersection body} of $B_p^n$ evaluated on a
  random vector from the sphere we infer a Central Limit Theorem for
  the norm induced by the intersection body evaluated on such a random
  vector (Corollary \ref{cor:intersection}).
\end{itemize}

Our approach relies on probabilistic formulas for volumes of sections
developed by Nayar and Tkocz \cite{MR4055953} (for the case $p \in
(0,2)$) and Chasapis, Nayar and Tkocz \cite{chasapis2021slicing} (for
$p > 2$). The former one allows us to represent the volume in terms of
expected determinants involving some auxiliary random variables, which
after applying certain reverse H\"older inequalities for polynomials
together with some geometric analysis allow to reduce the problem to
the Central Limit Theorem for $U$-statistics. The latter formula
relates the volume of a random section to the value of density at zero
for a randomly weighted sum of independent random variables, which
allows to apply conditionally an appropriate version of the Edgeworth
expansion for non-i.i.d. sequences. We note that it seems that the
lack of a version of Edgeworth expansion in higher dimension suitable
for our randomized setting is the main obstacle in extending our
results for $p > 2$ to general $d$. Since the investigation of
expansions of this form is rather distant from the main tools used in
this article we postpone it to future research. We refer the reader to
Remarks \ref{rem:Edgeworth-in-high-d} and \ref{rem:generic-CLT-etc}
for detailed comments concerning the difficulties in using Edgeworth
type results present in the literature and description of recent
developments on randomized Central Limit Theorems and Edgeworth
expansions.

\medskip

The organization of the article is as follows. After introducing the basic notation (Section \ref{sec:notation}) we state our main results (Section \ref{sec:main-results}). Section \ref{sec:tools} is devoted to introduction of various auxiliary results used in the main arguments. Finally Section \ref{sec:proofs-determinant} is devoted to the proof of results for $p \in (0,2)$ and general $d$, Section \ref{sec:proof-Edgeworth} to general $p<\infty$ and $d=1$, while in Section \ref{sec:cube} we sketch the proof for $p=\infty$ and $d=1$ and in Section \ref{sec:intersection} we prove the result on the intersection body of $B_p^n$.

\section{Notation}\label{sec:notation}

By $C$ we will denote universal constants, whereas $C_a$ will stand for constants depending only on the parameter $a$. In both cases the values of constants may change between occurrences (even within the same line).

For $x = (x_1,\ldots,x_n) \in \R^n$ and $p \in (0,\infty)$ we will denote $|x|_p = (|x_1|^p+\cdots+|x_n|^p)^{1/p}$. We set $|x|_\infty = \max_{i\le n}|x_i|$. As a function on $\R^n$ $|\cdot|_p$ is a norm for $p\ge 1$ and a quasi-norm for $p \in (0,1)$. We will also write $|\cdot|$ for $|\cdot|_2$.

For a random variable $X$, and $p \in \R$, by $\|X\|_p$ we will denote the $p$-th absolute moment of $X$, i.e. for $p\neq 0$, $\|X\|_p = (\E |X|^p)^{1/p}$, $\|X\|_0 = \exp(\E \log |X|)$ (in fact we will use $\|X\|_p$ for $p\neq 0$ only).

For a sequence of random variables $X_n$ and a sequence of positive real numbers $a_n$, we will write $X_n = o_\p(a_n)$ if $X_n/a_n$ converges in probability to zero and $X_n = \mathcal{O}_\p(a_n)$ if $X_n/a_n$ is bounded in probability. Let us also introduce similar notation, which despite being less standard, will allow us to shorten some formulas. We will write $X_n = \om(a_n)$ (resp. $X_n = \Om(a_n)$) if $X_n/a_n$ converges to zero (resp. is bounded) in \emph{all} spaces $L_q$ for $q > 0$ (with the subscript $m$ corresponding to \emph{moments}). We note that the speed of convergence (resp. implicit constants) may and usually will depend on $q$.

When dealing with independent random variables, e.g., $X,Y$ we will denote by $\E_X$, $\E_Y$ expectation with respect to just one of them (conditional expectation with respect to the other one). Variants of this standard convention will be used for larger families of independent random variables, the exact meaning will be either explicitly introduced or clear from the context.

By $\mathcal{L}(X)$ we will denote the law of a random variable $X$.

We will often work with multi-indices $\ii = (i_1,\ldots,i_d) \in [n]^d$. By $\rng$ we will denote the set of multi-indices with pairwise distinct coordinates. Similarly, for a set $I$  by $[n]^I$ (resp. $\Rng{I}$) we will denote the set of all (resp. all one-to-one) functions from $I$ to [n]. For $I\subset [d]$ and a multi-index $\ii \in [n]^d$ by $\ii_I$ we will denote $(i_\ell)_{\ell\in I} \in [n]^I$. Sometimes we will also use the notation $\ii_I$ to denote a stand-alone multi-index, writing for example $\sum_{\ii_I \in \Rng{I}} a_{\ii_I}$. For instance if $I = \{2,3\}$ this notation should be understood as $\sum_{1 \le i_2\neq i_3\le n} a_{i_2i_3}$.

\section{Main results}\label{sec:main-results}

We will now formulate our main results.

Recall that for $q \ge 1$ the Wasserstein distance $\Wass_q(P,Q)$ between two probability measures $P, Q$ on $\R^m$ is defined via the formula
\begin{displaymath}
\Wass_q(P,Q)^q = \inf_{{(X,Y)\colon}\atop {\mathcal{L}(X) = P, \mathcal{L}(Y) = Q}} \E |X-Y|^q.
\end{displaymath}

For $m=1$, the Wasserstein distance admits a representation as
\begin{displaymath}
  \Wass_q(P,Q)^q = \int_0^1 |F_P^{-1}(t) - F_{Q}^{-1}(t)|^q dt,
\end{displaymath}
where $F_P^{-1},F_Q^{-1}$ are generalized inverses of the cummulative distribution functions of $P,Q$ respectively.

It well is known (see, e.g., \cite[Theorem 7.12]{MR1964483}) that a sequence of probability measures $(P_n)$ converges in $\Wass_q$ to some measure $P$ if and only if it converges weakly and the $q$-th absolute moments of $P_n$ converge to the $q$-th absolute moment of $P$.

We are now ready to formulate the first result.

\begin{theorem}\label{thm:CLT}
Let $d$ be a positive integer and let $H_n$ be a random Haar distributed subspace of $\R^n$ of codimension $d$. For $p \in (0,2]$ and a positive integer $n$, define
\begin{align}\label{eq:ab}
  a_{p,d} &= \Big(\frac{\Gamma(1/p)}{\Gamma(3/p)}\Big)^{d/2}\frac{2^{d/2}\Gamma(1+1/p)^d }{\pi^{d/2}},\nonumber \\
  b_{p,d} &= (d+2)d\Big(\frac{\Gamma(5/p)}{\Gamma(1/p)} - 3\Big(\frac{\Gamma(3/p)}{\Gamma(1/p)}\Big)^{2}\Big)\Big(\frac{\Gamma(1/p)}{\Gamma(3/p)}\Big)^{d/2+2}\frac{2^{d/2-3}\Gamma(1+1/p)^d }{\pi^{d/2}}
\end{align}
Then the  sequence of random variables
\begin{displaymath}
n^{3/2}\Big(\frac{\Vol_{n-d}(H_n\cap B_p^n)}{\Vol_{n-d}(B_p^{n-d})} -   a_{p,d} - \frac{1}{n}b_{p,d}\Big)
\end{displaymath}
converges in distribution to a mean zero Gaussian random variable with variance
\begin{align}\label{eq:sigma-def}
  \Sigma_{p,d}^2 = \frac{2^{d-4}
d(d+5) \Gamma(1+1/p)^{2d} }{\pi^{d}}\cdot \Big(\frac{\Gamma(5/p)}{\Gamma(1/p)} - 3\Big(\frac{\Gamma(3/p)}{\Gamma(1/p)}\Big)^{2}\Big)^2\Big(\frac{\Gamma(1/p)}{\Gamma(3/p)}\Big)^{d+4}.
\end{align}
Moreover, for each $q\ge 1$ the convergence holds in the Wasserstein distance $\Wass_q$.
\end{theorem}

As a corollary to the  convergence in Wasserstein distances we immediately obtain the following refinement of the asymptotic expansion \eqref{eq:1st-order-B_pn-normalized} of $\E \Vol_{n-d}(B_p^n\cap H_n)$ .

\begin{cor}\label{cor:mean-approximation} In the setting of Theorem \ref{thm:CLT},
\begin{displaymath}
\E \Vol_{n-d}(B_p^n \cap H_n) = \frac{\Big(2\Gamma\Big(1+\frac{1}{p}\Big)\Big)^{n-d}}{\Gamma\Big(1+\frac{n-d}{p}\Big)}\Big(a_{p,d} + \frac{b_{p,d}}{n} + o(n^{-3/2})\Big).
\end{displaymath}
\end{cor}

\begin{remark}
The implicit constant in $o(n^{-3/2})$ in the above theorem depends on $p$ and $d$. In principle our proofs allow for obtaining estimates with explicit dependence on these parameters, also for moments of higher order, which would lead to a concentration of measure type result. However we do not pursue this direction.
\end{remark}

The method of proof of Theorem \ref{thm:CLT} is restricted to $p \in (0,2)$. In the special case of $d=1$ we can however extend the Central Limit Theorem to arbitrary $p \in (0,\infty)$.

\begin{theorem}\label{thm:d=1-p-arbitrary}
Let $H_n$ be a random Haar distributed subspace of $\R^n$ of codimension one. For $p \in (0,\infty)$ let $a_{p,1}, b_{p,1}, \Sigma_{p,1}^2$ be defined by
 \eqref{eq:ab} and \eqref{eq:sigma-def}, i.e.,
\begin{displaymath}
a_{p,1} = \frac{\sqrt{2}\Gamma(1+1/p)\Gamma(1/p)^{1/2}}{\sqrt{\pi} \Gamma(3/p)^{1/2}}, \;
b_{p,1} =   \frac{3\sqrt{2}\Gamma(1+1/p)}{8\sqrt{\pi}}\Big(\frac{\Gamma(1/p)}{\Gamma(3/p)}\Big)^{5/2}\Big(\frac{\Gamma(5/p)}{\Gamma (1/p)} - 3 \frac{\Gamma(3/p)^2}{\Gamma (1/p)^2}\Big)
\end{displaymath}
and
\begin{displaymath}
  \Sigma_{p,1}^2 =  \frac{3}{4\pi}\Big(\frac{\Gamma(1/p)}{\Gamma(3/p)}\Big)^5\Gamma(1+1/p)^2 \Big(\frac{\Gamma(5/p)}{\Gamma (1/p)} - 3 \frac{\Gamma(3/p)^2}{\Gamma (1/p)^2}\Big)^2.
\end{displaymath}
Then the sequence of random variables
\begin{displaymath}
  n^{3/2}\Big(\frac{\Vol_{n-1}(B_p^n\cap H)}{\Vol_{n-1}(B_{p}^{n-1})}- a_{p,1} - \frac{1}{n} b_{p,1}\Big)
  \end{displaymath}
converges in distribution to a mean zero Gaussian random variable with variance $\Sigma_{p,1}^2$.
\end{theorem}

Using the same method as in the proof of Theorem \ref{thm:d=1-p-arbitrary} we can obtain a result concerning more general sections of the cube $B_\infty^n$ by random hyperplanes. Note that while for fixed $p < \infty$ $\Vol_n(B_p^n) \to 0$ as $n \to \infty$, in the case of $p = \infty$ we have $\Vol_n(B_p^n) = 2^n$. For this reason in the limit theorem it is more convenient to normalize the volume of the section by $2^n$.

\begin{theorem}\label{thm:cube}
Let $H_n$ be a random Haar distributed subspace of $\R^n$ of codimension one and let $u$ be a unit vector normal to $H$.
For $x \in \R$ set
\begin{align*}
a(x) & = \sqrt{\frac{3}{2\pi}}e^{-\frac{3x^2}{2}},\\
b(x) & = -\frac{9\sqrt{3}}{20\sqrt{2\pi}}(3x^4-6x^2+1)e^{-\frac{3x^2}{2}}
\end{align*}
For every $x \in \R$,
the sequence
\begin{displaymath}
  n^{3/2}\Big(2^{-n}\Vol(B_\infty^n \cap (xu + H_n)) - a(x) - \frac{b(x)}{n}\Big)
\end{displaymath}
converges in distribution to a centered Gaussian variable with variance
\begin{displaymath}
  \Sigma(x)^2 = \frac{81}{100\pi}e^{-3x^2}(3x^4 - 6x^2+1)^2.
\end{displaymath}
\end{theorem}

Let us conclude the presentation of our results with a corollary concerning the intersection bodies of $B_p^n$. Recall that the radial function of a star-shaped body $C\subset \R^n $ in direction $u \in S^{n-1}$, is defined as $\rho_C(u) = \sup\{t> 0\colon u t \in C\}$.  The intersection body of a body $K$, introduced by Lutwak \cite{MR963487}, is a star-shaped body $\mathcal{I}K$ whose radial function is given by
$\rho_{\mathcal{I}K}(u) = \Vol_{n-1}(K\cap u^\perp)$, where $u^\perp$ is the hyperplane orthogonal to $u$. Thus our results for sections of codimension $d=1$ can be rephrased as a limit theorem for $\rho_{\mathcal{I}B_p^n}(\eta)$ for a random vector $\eta$ distributed uniformly on $S^{n-1}$. With some additional work one can infer from it a corollary concerning $\|\eta\|_{\mathcal{I}B_p^n} := \rho_{\mathcal{I}B_p^n}(\eta)^{-1}$ -- the Minkowski functional associated to $\mathcal{I}B_p^n$. We state it only for $p < \infty$ but clearly as similar result can be obtained for the cube.

\begin{cor}\label{cor:intersection}
Assume that $p \in (0,\infty)$ and let $\|\cdot\|_{\mathcal{I}B_p^n}$ be the Minkowski functional of the intersection body $\mathcal{I}B_p^n$. Let also $\eta$ be a random vector distributed uniformly on the sphere $S^{n-1}$, and let $a_{p,1}, b_{p,1}$ and $\Sigma_{p,1}^2$ be as in Theorem \ref{thm:d=1-p-arbitrary}.
Then as $n\to \infty$, the sequence of random variables
\begin{displaymath}
n^{3/2}\Big( \|\eta\|_{\mathcal{I}B_p^n} \Vol_{n-1}(B_p^{n-1}) - \frac{1}{a_{p,1}} + \frac{b_{p,1}}{na^2_{p,1}}\Big)
\end{displaymath}
converges in distribution to a Gaussian random variable with mean zero and variance
\begin{displaymath}
a_{p,1}^{-4} \Sigma_{p,1}^2= \frac{3\pi}{16}\Big(\frac{\Gamma(1/p)}{\Gamma(3/p)}\Big)^3\Gamma(1+1/p)^{-2} \Big(\frac{\Gamma(5/p)}{\Gamma (1/p)} - 3 \frac{\Gamma(3/p)^2}{\Gamma (1/p)^2}\Big)^2.
\end{displaymath}
\end{cor}

\section{Tools}\label{sec:tools}
\subsection{Stable random variables and the first volume formula}

Recall that for $\alpha \in (0,1)$, a positive random variable $Y$ is called a standard positive $\alpha$-stable random variable if $\E e^{-tY} = e^{-t^\alpha}$ for $t \ge 0$. In this case for $q < \alpha$,
\begin{align}\label{eq:stable-moments}
  \E Y^q = \frac{\Gamma(-q/\alpha)}{\alpha\Gamma(-q)}.
\end{align}
A positive $\alpha$-stable variable has a density which we will henceforth denote by $g_\alpha$. We refer to \cite{MR1280932,MR2166308} for basic information concerning stable random variables.

Our main tool will be the following formula for volumes of sections of $B_p^n$ due to Nayar and Tkocz \cite{MR4055953}. The article \cite{MR4055953} provides the proof in the case of $p=1$ with a remark that the same method works for general $p$. For reader's convenience and to provide the constants in the case $p\neq 1$ in Appendix \ref{app:proof-NT} we sketch the argument from \cite{MR4055953} in full generality.

\begin{theorem}\label{thm:N-T} Let $p \in (0,2)$ and let $H$ be a subspace of $\R^n$ of codimension d. Let $u_1,\ldots,u_d$ be an orthornormal basis in $H^\perp$ and let $v_1,\ldots,v_n$ be the columns of the matrix with rows $u_1,\ldots,u_d$. Then
\begin{align}\label{eq:N-T}
  \Vol_{n-d}(B_p^n\cap H) = \frac{2^{n}}{\pi^{d/2}}\frac{\Gamma\Big(1+ \frac{1}{p}\Big)^n}{\Gamma\Big(1+ \frac{n-d}{p}\Big)}\E\Big( \det\Big(\sum_{j=1}^n \frac{1}{W_j} v_j v_j^T\Big)\Big)^{-1/2},
\end{align}
where $W_j$'s are i.i.d. random variables with density proportional to $t\mapsto \frac{1}{\sqrt{t}}g_{p/2}(t)$.
\end{theorem}

\begin{cor}\label{cor:N-T}
In the setting of Theorem \ref{thm:N-T},
\begin{displaymath}
  \frac{\Vol_{n-d}(B_p^n\cap H)}{\Vol_{n-d}(B_p^{n-d})} = \frac{2^{d}\Gamma\Big(1+ \frac{1}{p}\Big)^d}{\pi^{d/2}} \E\Big( \det\Big(\sum_{j=1}^n \frac{1}{W_j} v_j v_j^T\Big)\Big)^{-1/2}.
\end{displaymath}
\end{cor}

In the sequel we will need a formula for moments given in the next lemma.
\begin{lemma}\label{le:W-moments} For $\alpha \in (0,1)$ let $W$ be a random variable with density proportional to $t\mapsto \frac{1}{\sqrt{t}}g_{\alpha}(t)$. Then for $q < \alpha + 1/2$,
\begin{displaymath}
  \E W^q = \frac{\Gamma\Big(\frac{1-2q}{2\alpha}\Big)\Gamma\Big(\frac{1}{2}\Big)}{\Gamma\Big(\frac{1-2q}{2}\Big)\Gamma\Big(\frac{1}{2\alpha}\Big)}.
\end{displaymath}
\end{lemma}

\begin{proof} Let $Y$ be a standard positive $\alpha$-stable random variable. Noting that the density of $W$ equals to $\frac{1}{\sqrt{t} \E Y^{-1/2}} g_\alpha(t)$, we get
\begin{displaymath}
  \E W^q = \frac{\E Y^{q-1/2}}{\E Y^{-1/2}} = \frac{\Gamma\Big(\frac{1-2q}{2\alpha}\Big)\Gamma\Big(\frac{1}{2}\Big)}{\Gamma\Big(\frac{1-2q}{2}\Big)\Gamma\Big(\frac{1}{2\alpha}\Big)},
\end{displaymath}
where in the last equality we used \eqref{eq:stable-moments}.
\end{proof}

\subsection{The second volume formula}
In the proof of Theorem \ref{thm:d=1-p-arbitrary} we will use another formula for volumes of sections of $B_p^n$. We will use only the special case for codimension one, but we formulate it in full generality.
The case $d=1$ is stated in \cite[Corollary 13]{chasapis2021slicing}, the general case also follows from arguments presented therein.

\begin{theorem}\label{thm:2nd-volume-formula}
Let $p \in (0,\infty)$ and let $H$ be a subspace of $\R^n$ of codimension $d$. Let $u_1,\ldots,u_d$ be an orthonormal basis in $H^\perp$ and let $v_1,\ldots,v_n$ be the columns of the matrix with rows $u_1,\ldots,u_d$. Let moreover $Y_1,\ldots,Y_n$ be independent random variables with density $e^{-\beta_p^p|x|^p}$, where $\beta_p = 2\Gamma(1+1/p)$. Then
\begin{displaymath}
\frac{\Vol_{n-d}(B_p^n\cap H)}{\Vol_{n-d}(B_p^{n-d})} = f(0),
\end{displaymath}
where $f \colon \R^d \to [0,\infty)$ is the continuous version of the density of the $\R^d$-valued random vector $\sum_{i=1}^n v_i Y_i$.
\end{theorem}

\subsection{Uniform random subspaces and the uniform spherical distribution}

Let us now introduce basic facts concerning uniform random subspaces of $\R^n$ and uniform distribution on $S^{n-1}$, which we will need in the proofs. We refer the reader, e.g., to the classical monograph \cite{MR856576} by Milman and Schechtman for a more detailed account.

By $G_{n,k}$ we will denote the Grassmanian, i.e., the space of all $k$-dimensional subspaces of $\R^n$. It is endowed with the unique normalized Haar measure inherited from the natural action of the orthogonal group. Whenever we speak about a random $k$-dimensional subspace of $\R^n$, we mean a random element of $G_{n,k}$ distributed according to this measure. It follows from the uniqueness of Haar measure that $H$ is a random $k$-dimensional subspace if and only if its orthogonal complement $H^\perp$ is a random $(n-k)$-dimensional subspace.

Basic properties of Gaussian random variables imply that if $G_1,\ldots,G_k$  are i.i.d. standard Gaussian vectors in $\R^n$, then $\spa(G_1,\ldots,G_k)$ is a random $k$-dimensional subspace of $\R^n$. Moreover if $u_1,\ldots,u_k$ are obtained from $G_1,\ldots,G_k$ as a result of the Gram-Schmidt orthogonalization, then each of the random vectors $u_i$ is distributed uniformly on the sphere $S^{n-1}$.

An important role in our arguments will be played by integrability properties of the random vectors $u_i$, which will follow from the classical concentration inequality. It is a consequence of the isoperimetic inequality on the sphere due to Lev\'y \cite{MR0041346} and Schmidt \cite{MR28600,MR34044}. Its crucial role in high dimensional geometry and probability was first observed by V. Milman in his proof of the Dvoretzky theorem \cite{MR0293374}.

\begin{theorem}\label{thm:spherical} Let $u$ be a random vector distributed uniformly on the sphere $S^{n-1}$. Then for any 1-Lipschitz function $f\colon S^{n-1}\to \R$, and any $t > 0$,
\begin{displaymath}
  \p(|f(u) - \E f(u)| \ge t) \le 2\exp(-cn t^2),
\end{displaymath}
where $c$ is a universal constant.
\end{theorem}

\subsection{Reverse H\"older inequalities for polynomials}
We will also need comparison of moments for polynomials in independent random variables. Historically results of this type (for positive moments) appeared first for Gaussian and Rademacher variables in the work by Nelson \cite{N} and Bonami \cite{Bon} in the context of hypercontractivity of Markov semigroups. Later they were extended to more general variables in particular in \cite{MR931501,MR1085342}. We refer to the monographs \cite{MR1167198,MR1666908} for a detailed account.

Apart from comparison of positive moments we will need a result allowing us to compare negative moments with the first moment for tetrahedral polynomials with nonnegative coefficients (recall that a multivariate polynomial is called tetrahedral if it is affine in each of the variables). We have not been able to find results in this spirit in the literature, we provide them in Lemma \ref{le:hypercontractivity-negative} and Corollary \ref{cor:moment-comparison-negative} below. We remark that certain reverse H\"older inequalities for negative moments of positive functions in the discrete setting have been considered in \cite{MR3061780}. Inequalities for general polynomials and log-concave measures were obtained in \cite{MR1839474}. For our purposes we need however stronger estimates in a simpler setting of independent, positive random variables.

Let $\kappa > \rho > 1$. We say that a random variable $X$ is $(\kappa,\rho)$ hypercontractive with constant $\gamma$ if for every $a,b \in \R$
\begin{align}\label{eq:hypercontractivity-def}
  \|a+b\gamma X\|_\kappa \le \|a+bX\|_\rho.
\end{align}

The following theorem is the real valued case of \cite[Proposition 3.2]{MR931501}. We remark that the original formulation is more general as it involves hypercontractivity for Hilbert space valued coefficients. Moreover it provides explicit formulas for the value of the constant $\gamma$, which we will omit, as it will not be needed in the sequel.

\begin{prop}\label{thm:hypercontractivity}
If $1<\rho\le 2\le \kappa$ and $X$ is a centered random variable such that $\|X\|_\kappa < \infty$, then $X$ is $(\kappa,\rho)$-hypercontractive.
\end{prop}

As is well known (see, e.g., \cite[Theorem 2.5]{MR931501} or \cite[formula (1.4)]{MR1085342}) hypercontractivity implies comparison of moments for tetrahedral polynomials.

\begin{prop}\label{thm:moment-comparison}
Let $1<\rho\le \kappa$. If $X_1,\ldots,X_n$ are independent random variables, $(\kappa,\rho)$-hypercontractive with the same constant $\gamma$, then for all tetrahedral polynomials $Q$ in $n$ variables
\begin{displaymath}
  \|Q(\gamma X_1,\ldots,\gamma X_n)\|_\kappa \le \|Q(X_1,\ldots,X_n)\|_\rho.
\end{displaymath}
\end{prop}

\begin{cor}\label{cor:moment-comparison}
Let $\kappa \ge 2$ and let $X_1,\ldots,X_n$ be i.i.d. copies of a random variable $X$ such that $\E |X|^\kappa < \infty$. Then for every $\rho > 0$, there exists a constant $K$ depending only on $d,\kappa,\rho$ and the law of $X$ such that for every tetrahedral polynomial of degree $d$ in $n$ variables,
\begin{displaymath}
  \|Q(X_1,\ldots,X_n)\|_\kappa \le K\|Q(X_1,\ldots,X_n)\|_\rho.
\end{displaymath}
\end{cor}

\begin{proof}
Since $Q(X_1,\ldots,X_n)$ can be written as a tetrahedral polynomial in the variables $X_1-\E X,\ldots,X_n - \E X$ we may and will assume that $X$ is centered. Moreover we can assume that $q \in (1,2)$. Indeed, let $Z$ be a random variable such that for some $\kappa > \rho>0$ and a constant $D$,
\begin{displaymath}
  \|Z\|_\kappa \le D\|Z\|_\rho,
\end{displaymath}
and let $\rho' \in (0,\rho)$. Then, by H\"older's inequality applied with exponents $t, t/(t-1)$ for $t = (\kappa-\rho')/(\kappa- \rho)>1$, we get
for $r = \rho'/(\rho t) < 1$,
\begin{multline*}
  \E Z^\rho = \E Z^{\rho r}Z^{\rho(1-r)} \le (\E Z^{\rho r t})^{1/t}(\E Z^{\rho (1-r)t/(t-1)})^{1 - 1/t} = (\E Z^{\rho'})^{1/t} (\E Z^\kappa)^{1-1/t}\\
  \le  (\E Z^{\rho'})^{1/t} D^{\kappa(1-1/t)}(\E Z^\rho)^{\kappa (1-1/t)/\rho} = (\E Z^{\rho'})^{1/t} D^{\rho(1-r)} (\E Z^\rho)^{1-r},
\end{multline*}
from which one obtains $\|Z\|_\rho \le D^{1/r-1}\|Z\|_{\rho'}$ and as a consequence $\|Z\|_\kappa \le D^{1/r}\|Z\|_{\rho'}$. In what follows we will thus assume that $q \in (1,2)$.

Write $Q = Q_0+\ldots+Q_d$, where $Q_i$ is the homogeneous part of $Q$ of degree $i$. By Propositions \ref{thm:hypercontractivity} and \ref{thm:moment-comparison}
\begin{displaymath}
  \|Q(X_1,\ldots,X_n)\|_\kappa \le \sum_{i=0}^d \|Q_i(X_1,\ldots,X_n)\|_\kappa \le K\sum_{i=0}^d \|Q_i(X_1,\ldots,X_n)\|_\rho
\end{displaymath}
for some constant $K$ as in the statement of the corollary.

By  \cite[Proposition 1.2]{MR3052405} (see also \cite[Lemma 2]{MR893914}) the right hand side above is bounded from above by $C_d K\|Q(X_1,\ldots,X_n)\|_\rho$, which ends the proof.
\end{proof}

Above we have considered only tetrahedral polynomials, however in the Gaussian case, thanks to infinite divisibility, each polynomial can be approximated by a tetrahedral one. Hypercontractivity of Gaussian variables in the sense of \eqref{eq:hypercontractivity-def} goes back to the seminal work of Nelson, Bonami and Gross related to hypercontractivity for Markov semigroups. As a consequence of their results we have the following theorem (we remark that explicit constants are known, but as they are not needed for our proofs we state the result in a simplified version).

\begin{theorem}\label{thm:Gaussian-hypercontractivity}
Let $G$ be a standard Gaussian vector in $\R^n$. For each $\kappa,\rho > 0$ and every polynomial $Q\colon \R^n\to \R$ of degree at most $d$,
\begin{displaymath}
  \|Q(G)\|_\kappa \le C_{\kappa,\rho,d}\|Q(G)\|_{\rho}.
\end{displaymath}
\end{theorem}

So far we have dealt only with positive moments, however a crucial step in the proof of Theorem \ref{thm:CLT} will be based on comparison of the first moment with negative moments for positive polynomials in independent random variables. Its proof relies on a similar idea as presented above.

\begin{lemma}\label{le:hypercontractivity-negative} Let $X$ be a positive random variable such that $\E X < \infty$ and $\E X^{-\rho} < \infty$ for some $\rho > 0$. Then for any $\kappa \in (0,\rho]$ and any $a \ge 0$,  $b \ge 0$,
\begin{displaymath}
  \E \frac{1}{(a+b X)^\kappa} \le \frac{1}{(a + \sigma b \E X)^\kappa},
\end{displaymath}
where $\sigma = \|X\|_{-\rho}/{\|X\|_1}$.
\end{lemma}

\begin{proof}
We may assume without loss of generality that $a,b > 0$. Consider the function $\varphi\colon \R_+ \to \R_+$ given by
\begin{displaymath}
  \varphi(z) = \frac{1}{(1 + z^{-1/\rho})^\kappa}.
\end{displaymath}
We have
\begin{displaymath}
  \varphi''(z) = -\frac{\kappa (z^{-1/\rho} + 1)^{-\kappa} (-\kappa + \rho z^{1/\rho} + z^{1/\rho} + \rho)}{\rho^2 z^2 (z^{1/\rho} + 1)^2} \le 0
\end{displaymath}
for $z > 0$. Thus $\varphi''$ is concave on $\R_+$. Setting $Z = X^{-\rho}$, by Jensen's inequality we get
\begin{multline*}
  \E \frac{1}{(a+b X)^\kappa} = \frac{1}{a^\kappa} \E \frac{1}{( 1 + ((a/b)^{\rho}Z)^{-1/\rho})^\kappa} \\
  = \frac{1}{a^\kappa} \E \varphi((a/b)^{\rho} Z) \le \frac{1}{a^\kappa} \varphi((a/b)^{\rho} \E Z) =  \frac{1}{(a + b \|X\|_{-\rho})^\kappa} = \frac{1}{(a + \sigma b \E X)^\kappa}.
\end{multline*}
\end{proof}

\begin{cor}\label{cor:moment-comparison-negative}
In the setting of Lemma \ref{le:hypercontractivity-negative}, let $X_1,\ldots,X_n$ be i.i.d. copies of the random variable $X$ and let $(a_{i_1,\ldots,i_d})_{i_1,\ldots,i_d \in [n]}$ be an array of nonnegative numbers such that $a_{i_1,\ldots,i_d} = 0$ whenever there exists $l\neq m$ such that $i_l = i_m$. Let $Y = \sum_{i_1,\ldots,i_d = 1}^n a_{i_1,\ldots,i_d} X_{i_1}\cdots X_{i_d}$. Then
\begin{displaymath}
  \|Y\|_1 \le \sigma^{-d} \|Y\|_{-\kappa}.
\end{displaymath}
\end{cor}

\begin{proof}
We will prove by induction on $n$ that for any tetrahedral polynomial $Q$ in $n$ variables, with nonnegative coefficients,
\begin{align}\label{eq:induction-rev-hyp}
  \E \frac{1}{Q(X_1,\ldots,X_n)^\kappa} \le \frac{1}{(\E Q(\sigma X_1,\ldots,\sigma X_n))^\kappa},
\end{align}
from which the corollary follows by homogeneity of the $d$-linear form.

For $n=0$, \eqref{eq:induction-rev-hyp} is trivial, while for $n=1$ it reduces to the assertion of lemma \ref{le:hypercontractivity-negative}. Assume thus that \eqref{eq:induction-rev-hyp} holds for $0,1,\ldots,n-1$. Any tetrahedral polynomial in $n$ variables with nonnegative coefficients is of the form $Q(x_1,\ldots,x_n) = x_n Q_1(x_1,\ldots,x_{n-1}) + Q_2(x_1,\ldots,x_{n-1})$ where $Q_1,Q_2$ are tetrahedral and also have nonnegative coefficients. Thus, by Lemma \ref{le:hypercontractivity-negative} applied conditionally on $X_1,\ldots,X_{n-1}$ and the induction assumption we obtain
\begin{eqnarray*}
\E \frac{1}{Q(X_1,\ldots,X_n)^\kappa} &=& \E_{X_1,\ldots,X_{n-1}} \E_{X_n} \frac{1}{(X_n Q_1(X_1,\ldots,X_{n-1}) + Q_2(X_1,\ldots,X_{n-1}))^\kappa}\\
&\le& \E_{X_1,\ldots,X_{n-1}} \frac{1}{((\E_{X_n} \sigma X_n) Q_1(X_1,\ldots,X_{n-1}) + Q_2(X_1,\ldots,X_{n-1}))^\kappa}\\
&\le &\frac{1}{\Big(\E_{X_1,\ldots,X_{n-1}} ((\E_{X_n} \sigma X_n) Q_1(\sigma X_1,\ldots,\sigma X_{n-1}) + Q_2(\sigma X_1,\ldots,\sigma X_{n-1})\Big)^\kappa}\\
&=&\frac{1}{(\E Q(\sigma X_1,\ldots,\sigma X_n))^\kappa},
\end{eqnarray*}
where we used the fact that $(\sigma \E_{X_n}X_n) Q_1(x_1,\ldots,x_{n-1}) + Q_2(x_1,\ldots,x_{n-1})$ is tetrahedral, which allows us to use the induction assumption.
\end{proof}

\subsection{$U$-statistics}\label{sec:U-statistics}

Let $X,X_1,X_2,\ldots$ be a sequence of i.i.d. random variables with values in some measurable space $(\mathcal{S},\mathcal{F})$ and let $h\colon \mathcal{S}^d\to \R$ be a measurable function. The $U$-statistics of order $d$ and kernel $h$ based on the sequence $\{X_i\}_{i\ge 1}$ are random variables defined as
\begin{displaymath}
  U_n(h) = U_n^{(d)}(h) = \frac{(n-d)!}{n!} \sum_{\ii \in \rng} h(X_{i_1},\ldots,X_{i_d}),\; n \ge d.
\end{displaymath}
Define also the unnormalized sum
\begin{displaymath}
  S_n(h) = S_n^{(d)}= \sum_{\ii \in \rng} h(X_{i_1},\ldots,X_{i_d}) = \frac{n!}{(n-d)!} U_n(h).
\end{displaymath}
We will also use the above definitions for $d=0$ interpreting $\mathcal{S}^0$ as a one element set and identifying  functions from $\mathcal{S}^0$ into $\R$ with constants. Thus according to this convention for any constant $h$, $S_n^{(0)}(h) = U_n^{(0)}(h) = h$.

Going back to general $d$, note that by modifying the kernel $h$ we may assume without loss of generality that it is invariant under permutation of the arguments. In the subsequent part we will work under this assumption.

$U$-statistics appeared for the first time in the 1940s in the work by Halmos \cite{MR15746} and Hoeffding \cite{MR26294} in the context of unbiased estimation. Since then they have found applications, e.g., as higher order terms in Taylor expansions of smooth statistics \cite{MR1666908}, or in random graph theory \cite{MR1474726}. In a geometric context they were used in \cite{MR1092403} to prove a CLT for volumes of Minkowski sums of random convex sets (see also the monograph \cite{MR3751326}) and in \cite{MR3230006} in the proof of the CLT for volumes of low dimensional random projections of the cube. The reason behind the appearance of $U$-statistics in these references and our proof are quite different. In our case they show up as some conditional variances in the one dimensional Taylor expansion of functions appearing in the volume formula of Theorem \ref{thm:N-T}, whereas in the cited references they are directly related to mixed volumes and the zonotope volume formula. Yet another line of research involving a somewhat related but different notion of Poisson $U$-statistics has been pursued, e.g., in \cite{MR3161465,MR3215537}, where the authors among other examples  consider applications to intrinsic volumes of sets arising from the Poisson flat process. These results are of a different flavour than ours, they are related to stochastic analysis on the Poisson space and concern asymptotic behaviour of random sets in a fixed dimension, when the intensity of the underlying Poisson process tends to infinity.

Below we describe basic properties of $U$-statistics, referring for a more detailed description of their asymptotic theory to the monographs \cite{MR1474726, MR1666908,MR1075417} .

The kernel $h$ (and the corresponding $U$-statistic) is called canonical if $\E h(X,x_2,\ldots,x_d) = 0$ for all $x_2,\ldots,x_d \in \mathcal{S}$. One of the basic tools in the theory of $U$-statistics is the Hoeffding decomposition, which allows to represent any $U$-statistic based on a square integrable kernel of mean zero as a sum of uncorrelated canonical ones. Let us now briefly describe how it works. We will use the notation $\mu_1\otimes\cdots \otimes \mu_d f = \int_{\mathcal{S}^n} f(x_1,\ldots,x_d)\mu_1(dx_1)\cdots \mu_d(x_d)$, where $\mu_1,\ldots,\mu_d$ are signed measures on $\mathcal{S}$. For $k=0,\ldots,d$ we define the $k$-th Hoeffding projection of $h$ as
\begin{displaymath}
  \pi_k h (x_1,\ldots,x_k) = (\delta_{x_1} - P)\otimes \cdots\otimes (\delta_{x_k} - P)\otimes P^{\otimes (d-k)} h,
\end{displaymath}
where $P$ is the law of $X$ and $\delta_x$ stands for the Dirac mass at $x$.
For instance
\begin{align*}
\pi_0 h &= \E h(X_1,\ldots,X_d),\\
\pi_1 h(x_1) & = \E h(x_1,X_2,\ldots,X_d) - \E h(X_1,\ldots,X_d),\\
\pi_2 h(x_1,x_2) &= \E h(x_1,x_2,X_3,\ldots,X_d) - \E h(x_1,X_2,\ldots,X_d) - \E  h(x_2,X_2,\ldots,X_d) + \E h(X_1,\ldots,X_d).
\end{align*}
One checks that for $k > 0$ the kernel $\pi_k h$ is canonical. Moreover
\begin{displaymath}
  U_n^{(d)}(h) = \sum_{k=0}^d \binom{d}{k} U_n^{(k)}(\pi_k h)
\end{displaymath}
or in terms of $S_n^{(k)}$,
\begin{align}\label{eq:Hoeffding-Sn}
  S_n^{(d)}(h) = \sum_{k=0}^d \binom{d}{k} \frac{(n-k)!}{(n-d)!} S_n^{(k)}(\pi_k h).
\end{align}

Canonical $U$-statistics can be thus considered building blocks for the general case.  Their advantage stems from the following elementary estimate for their second moment. If $h$ is canonical then
\begin{displaymath}
    \E |S_n^{(d)}(h)|^2 = \frac{n!d!}{(n-d)!} \E h(X_1,\ldots,X_d)^2.
\end{displaymath}
Indeed the left-hand side above equals
\begin{displaymath}
  \sum_{\ii\in \rng} \sum_{\jj \in \rng} \E h(X_{i_1},\ldots,X_{i_d}) h(X_{j_1},\ldots,X_{j_d})
\end{displaymath}
and using the canonicity and symmetry one can easily see that the expectations do not vanish only if $\ii$ and $\jj$ differ just by a permutation of coordinates, in which case they  equal $\E h^2(X_1,\ldots,X_d)$.
In particular, since $\pi_k h$ is a canonical kernel of order $k$, we obtain that for $k > 0$,
\begin{align}\label{eq:canonical-second-moment}
  \E |S_n^{(k)}(\pi_k h)|^2 = \frac{n!k!}{(n-k)!} \E h(X_1,\ldots,X_k)^2.
\end{align}

It is also straightforward to check that $S_n^{(k)}(\pi_k h)$, $k=0,\ldots,n$ are uncorrelated. In particular \eqref{eq:Hoeffding-Sn} for $n=d$ together with \eqref{eq:canonical-second-moment} give
\begin{align}\label{eq:Hoeffding-contraction property}
\E h(X_1,\ldots,X_d)^2 = \sum_{k=0}^d \binom{d}{k} \E (\pi_k h(X_1,\ldots,X_k))^2.
\end{align}

\subsection{Hermite polynomials}
Recall that Hermite polynomials are orthogonal polynomials with respect to the standard Gaussian measure. There are many conventions concerning their normalization, we will use the probabilistic one and define for $n\in \N$,
\begin{displaymath}
  H_n(x) = (-1)^n e^{x^2/2} \frac{d^n}{dt^n} e^{-x^2/2}.
\end{displaymath}

Thus in particular
\begin{align}\label{eq:first-Hermites}
H_0(x) & = 1,\nonumber\\
H_1(x) & = x,\nonumber\\
H_2(x) & = x^2 - 1,\nonumber\\
H_3(x) & = x^3 - 3x,\nonumber\\
H_4(x) & = x^4 - 6x^2 + 3.
\end{align}

We refer, e.g., to the monograph \cite{MR1474726} for the presentation of various probabilistic properties and applications of Hermite polynomials. In our setting only low degree polynomials will appear, in two different settings, to simplify calculation of expectations of Gaussian polynomials and as ingredients in the Edgeworth expansion.

\subsection{Cumulants}
Let us now briefly recall the basic properties of cumulants of random variables.

Let $X$ be a real valued random variable. For $n\ge 1$ the cumulants $\ka_n = \ka_n(X)$ are defined as
\begin{displaymath}
  \ka_n = \i^{-n}\frac{d^n}{dt^n}\log \E e^{\i t X}\Big|_{t=0}.
\end{displaymath}
Clearly $\ka_n$ are well defined if $X$ has all moments. Moreover we have the identity
\begin{align}\label{eq:cumulants-moments}
  \ka_n = \sum_{{1\cdot r_1+\cdots+n\cdot r_n = n }\atop{r_1,\ldots,r_n\in \N}}\frac{(-1)^{r_1+\cdots+r_n-1}(r_1+\cdots+r_n-1)!n!}{(1!)^{r_1}(2!)^{r_2}\cdots (n!)^{r_n}r_1!\cdots r_n!}(\E X)^{r_1}(\E X^2)^{r_2}\cdots (\E X^n)^{r_n}.
\end{align}
In particular
\begin{align}\label{eq:low-cumulants}
\ka_1 & = \E X,\nonumber \\
\ka_2 & = \E X^2 - (\E X)^2 = \Var(X),\nonumber\\
\ka_3 & = 2(\E X)^3 - 3\E X \E X^2 + \E X^3 = \E(X - \E X)^3,\nonumber\\
\ka_4 & = \E (X-\E X)^4 - 3 (\E(X-\E X)^2)^2.
\end{align}

The other important properties of cumulants are
\begin{displaymath}
  \ka_n(X+t) = \ka_n(X),\; \ka_n(tX) = t^n \ka_n(X)
\end{displaymath}
for $t \in \R$ and
\begin{displaymath}
  \ka_n(X_1+\cdots+X_m) = \ka_n(X_1)+\cdots+\ka_n(X_m)
\end{displaymath}
if $X_1,\ldots,X_m$ are independent random variables.

It is also easy to see that odd cumulants of a symmetric random variable vanish.

\subsection{Edgeworth expansion}\label{sec:Edgeworth}
We are also going to need basic results on Edgeworth expansion of the density for sums of independent non-identically distributed random variables, which provides the correction to the local Central Limit Theorem. As references for this topic we propose the classical monographs \cite{MR0388499,MR0436272}.

Consider a sequence $X_1,X_2,\ldots$ of independent centered random variables. Let $B_n = \sum_{i=1}^n \E X_i^2$ and for integers $k \ge 2$ and $n\ge 1$ define the number
\begin{align}\label{eq:lambda-definition}
  \lambda_{kn} = \frac{n^{(k-2)/2}}{B_n^{k/2}}\sum_{i=1}^n \ka_k(X_i)
\end{align}
and a function $q_{kn}\colon \R \to \R$,
\begin{displaymath}
  q_{kn}(x) = \frac{1}{\sqrt{2\pi}}e^{-x^2/2}\sum_{{1\cdot r_1+\cdots+k\cdot r_k = k }\atop{r_1,\ldots,r_k\in \N}} H_{k+2(r_1+\cdots+r_k)}(x)\prod_{l=1}^k\frac{1}{r_l!}\Big(\frac{\lambda_{l+2,n}}{(l+2)!}\Big)^{r_l}.
\end{displaymath}
In particular
\begin{align}\label{eq:terms-of-expansion}
q_{0n}(x) &= \frac{1}{\sqrt{2\pi}}e^{-x^2/2},\nonumber\\
q_{1n}(x) &= \frac{1}{\sqrt{2\pi}}e^{-x^2/2}H_3(x)\frac{\lambda_{3,n}}{6},\nonumber\\
q_{2n}(x) &= \frac{1}{\sqrt{2\pi}}e^{-x^2/2}\Big(H_4(x)\frac{\lambda_{4,n}}{24}+ H_6(x)\frac{\lambda_{3,n}^2}{72}\Big),\nonumber\\
q_{3n}(x) &= \frac{1}{\sqrt{2\pi}}e^{-x^2/2}\Big(H_5(x)\frac{\lambda_{5,n}}{120} + H_7(x)\frac{\lambda_{3,n}\lambda_{4,n}}{144} + H_9(x)\frac{\lambda_{3,n}^3}{1296}\Big).
\end{align}

The Edgeworth type theorem we are going to use is the following special case ($l=1$) of \cite[Theorem 7, p. 175]{MR0388499} .

\begin{theorem}\label{thm:Edgeworth}
Let $X_1,X_2,\ldots$ be independent random variables with mean zero. Denote $B_n = \sum_{i=1}^n \E X_i^2$. Let $K \ge 3$ be an integer. Assume that
\begin{itemize}
\item[(i)] $\liminf_{n\to \infty} \frac{B_n}{n} > 0$, \; $\limsup_{n\to \infty} \frac{1}{n}\sum_{i=1}^n \E |X_i|^K < \infty$,
\item[(ii)] $\lim_{n\to \infty} \frac{1}{n}\sum_{i=1}^n \E |X_i|^K \ind{|X_i| > n^\tau}  = 0$ for some positive $\tau < 1/2$,
\item[(iii)] for $\alpha = (K-1)/2$ and every fixed $\varepsilon > 0$
\begin{align}\label{eq:3rd-condition}
  \lim_{n\to \infty} n^\alpha \int_{|t|>\varepsilon} \prod_{i=1}^n \E e^{\i t X_i}dt = 0.
\end{align}
Then for all $n$ sufficiently large the random variable $S_n = B_{n}^{-1/2}\sum_{i=1}^n X_i$ admits a density $f_n$ such that
\begin{align}\label{eq:Edgeworth-expansion}
  f_n(x) = \frac{1}{\sqrt{2\pi}}e^{-x^2/2} + \sum_{k=1}^{K-2}\frac{q_{kn}(x)}{n^{k/2}} + o\Big(\frac{1}{n^{(K-2)/2}}\Big)
\end{align}
uniformly in $x\in \R$.
\end{itemize}
\end{theorem}

To apply the above theorem we will need the following lemma (a corollary to \cite[Lemma 10]{MR0388499} and the remark following its proof (p. 174)).

\begin{lemma}\label{le:Edgeworth-conditions} Let $X_1,X_2,\ldots$ be a sequence of independent centered random variables. Assume that there exists a set $\mathcal{I}$ of positive integers and positive constants $\lambda,\delta,C,r$ such that
\begin{displaymath}
\liminf_{n\to \infty} \frac{|\mathcal{I}\cap [n]|}{n^\lambda} > 0,
\end{displaymath}
and for all $n \in \mathcal{I}$ and $|t| > R$,
\begin{align}\label{eq:Edgeworth-lemma-2nd-condition}
  |\E e^{\i t X_n}| \le \frac{C}{|t|^\delta}.
\end{align}
Then the condition \eqref{eq:3rd-condition} is satisfied for any $\alpha > 0$. Moreover the condition \eqref{eq:Edgeworth-lemma-2nd-condition} is satisfied for some $C,\delta,R > 0$ if the variables $X_n, n\in \mathcal{I}$, have densities with the variation bounded by a common constant.
\end{lemma}

\section{Proof of Theorem \ref{thm:CLT}}\label{sec:proofs-determinant}

Let us now give the proof of Theorem \ref{thm:CLT}. It will be split into several steps presented in separate sections. To simplify the notation we will often suppress the dependence of random variables on the dimension $n$. It should be remembered that the asymptotic notation such as $o(1)$, $O(n)$ etc. concerns $n\to \infty$. The implicit constants may depend on $d,p$ and some other parameters (independent of $n$), e.g., the order of moments considered. Recall also the notation $\Om(\cdot)$, $\om(\cdot)$ introduced in Section \ref{sec:notation}.

\subsection{Preliminary steps}
A random subspace of $\R^n$ of dimension $d$ may be obtained as $\spa(G_1,\ldots,G_d)$, where $G_i = G_i^{(n)} = (g_{i1},\ldots,g_{in})$, $i=1,\ldots,d$ are i.i.d. standard Gaussian vectors in $\R^n$ (i.e., $g_{ij}$, $i \in [d], j\in [n]$ are i.i.d. Gaussian variables of mean zero and variance one). If we define $u_1,\ldots,u_d$ as the result of the Gram-Schmidt orthogonalization applied to $G_1,\ldots,G_d$ then $H = \spa(G_1,\ldots,G_d)^\perp = \spa(u_1,\ldots,u_d)^\perp$ is a uniform subspace of codimension $d$ in $\R^n$. More explicitly we have
\begin{align*}
u_i = \frac{G_i - P_{i-1} G_i}{|G_i - P_{i-1} G_i|},
\end{align*}
where for $i = 0,\ldots,d-1$, $P_i \colon \R^n \to \R^n$ is the orthogonal projection on the random subspace $\spa(G_1,\ldots,G_i)$. Note that with probability one this subspace is of dimension $i$. Moreover, the random operator $P_{i-1}$ is stochastically independent of the random vectors $G_i,\ldots,G_d$.

Observe also that each $u_i$ is distributed uniformly on the sphere $S^{n-1}$. Note that the second moment of a single coordinate of $u_i$ is equal to $1/\sqrt{n}$. Therefore, by the concentration inequality on the sphere (Theorem \ref{thm:spherical}) and the union bound, for any $t \ge 1$,
\begin{align}\label{eq:l-infinity-first-bound}
  \p\Big(\max_{i\le d} \|u_i\|_\infty \ge Ct\frac{\sqrt{\log(n)}}{\sqrt{n}}\Big) \le 2nd\exp(-ct^2 \log n)
\end{align}
As a consequence we obtain
\begin{align}\label{eq:l-infinity-bound}
  \p\Big(\max_{i\le d} \|u_i\|_\infty \ge C \frac{\sqrt{\log n}}{\sqrt{n}}\Big) = o(1).
\end{align}
Moreover, for any $q > 0$, \eqref{eq:l-infinity-first-bound} and integration by parts give
\begin{align}\label{eq:l-infinity-moments}
\Big\|\max_{i\le d} \|u_i\|_\infty \Big\|_q\le C_{d,q} \frac{\sqrt{\log n}}{\sqrt{n}}
\end{align}
in other words $\max_{i\le d} \|u_i\|_\infty = \Om(\sqrt{(\log n)/n})$ (recall the notation $\Om, \om$ introduced in Section \ref{sec:notation}).

Let $W_1,\ldots,W_n$ be i.i.d. random variables with density proportional to $\frac{1}{\sqrt{x}}g_{p/2}(x)$ and $X_i = W_i^{-1}$. We will assume that the family $(W_i)_i$ is stochastically independent of $(g_{ij})_{ij}$ and we will denote integration with respect to them by $\E_W$ and $\E_G$ respectively.

Going back to random subspaces, denote by A the $d\times n$ matrix with rows $u_1,\ldots,u_d$, and let $v_1,\ldots,v_n$ be its columns (thus $v_i$'s are random vectors in $\R^d$). Let also $\widetilde{A}$ be the matrix with columns $\sqrt{X_i} v_i$. Then
\begin{displaymath}
  \sum_{j=1}^n X_j v_j v_j^T = \tilde{A}\tilde{A}^T
\end{displaymath}
and using the Cauchy--Binet formula we get
\begin{align}\label{eq:determinant}
 Y_n := \det\Big(\sum_{j=1}^n X_j v_j v_j^T\Big) = \sum_{\ii \in \rng} a_\ii X_{i_1}\cdots X_{i_d},
\end{align}
where for $\ii = (i_1,\ldots,i_d) \in \rng$, $a_\ii = \frac{1}{d!}(\det A_\ii)^2$, with $A_\ii$ being the $d\times d$ matrix with columns $ v_{i_1},\ldots,v_{i_d}$ (we choose to write the quantities of interest in terms of summation over the whole set $\rng$ rather than just over increasing sequences to simplify bookkeeping of indices which will appear in the sequel). Another application of the Cauchy-Binet formula gives
\begin{align} \label{eq:sum-of-a's}
\sum_{\ii \in \rng} a_\ii = \det\Big(\sum_{j=1}^n v_jv_j^T\Big) = \det(\id) = 1.
\end{align}

Moreover, using the multilinearity of the determinant and the fact that it vanishes for matrices which are not of full rank together with the definition of the vectors $u_j$ and $v_j$ one can expand $\det A_\ii$ along the steps of the Gram-Schmidt procedure and see that
\begin{align}\label{eq:def-a_ii}
a_\ii = \frac{1}{d!} \frac{1}{\prod_{\ell=1}^d |G_\ell - P_{\ell-1} G_\ell|^2} \det(\Gamma_\ii)^2,
\end{align}
where $\Gamma_\ii$ is the $d\times d$ matrix with columns $\Gamma_{i_1},\ldots,\Gamma_{i_d}$ with $\Gamma_i = (g_{ji})_{j \in [d]}$.

Moreover, by \eqref{eq:l-infinity-bound}, with probability tending to one as $n\to \infty$,
\begin{align}\label{eq:2nd-l-infinity-bound}
  \max_{\ii \in \rng} a_\ii \le C_d \frac{(\log n)^d}{n^d}
\end{align}
whereas \eqref{eq:l-infinity-moments} gives
\begin{align}\label{eq:2nd-l-infinity-bound-moments}
  \max_{\ii \in \rng} a_\ii =\Om\Big(\frac{(\log n)^d}{n^d}\Big).
\end{align}

Observe that thanks to  Lemma \ref{le:W-moments} (applied with $\alpha = p/2$) we have $\E X_i^\rho < \infty$ for all $\rho> 0$ and $\E X_i^{-\rho} < \infty$ for $\rho < p/2 + 1/2$.
Note that $p/2 +1/2 > 1/2$.  Using Corollary \ref{cor:moment-comparison-negative} together with \eqref{eq:sum-of-a's} we obtain that there exists $\rho=\rho_p>1/2$ and a constant $C_{p,d}$, depending only on $p$ and $d$, such that
\begin{align}\label{eq:negative-Holder-CLT}
  \E_X \frac{1}{(\sum_{\ii\in\rng} a_\ii X_{i_1}\cdots X_{i_d})^\rho} \le C_{p,d} \frac{1}{(\sum_{\ii\in \rng} a_\ii)^\rho} = C_{p,d}.
\end{align}
(Note that thanks to \eqref{eq:sum-of-a's} the bound on the random quantity on the left-hand side becomes deterministic.)

Similarly, by Corollary \ref{cor:moment-comparison}  we get that for all $r > 1$ there exists $C_{p,d,r}$ such that for every tetrahedral polynomial $Q$ of degree at most $d$,
\begin{align}\label{eq:positive-Holder-CLT}
   \|Q(X_1,\ldots,X_n)\|_r \le C_{p,d,r}\|Q(X_1,\ldots,X_n)\|_1.
\end{align}
By Corollary \ref{cor:N-T},
\begin{align}\label{eq:volume-representation-proof-of-CLT}
  \frac{\Vol_{n-d}(H_n\cap B_p^n)}{\Vol_{n-d}(B_p^{n-d})} = \frac{2^{d}\Gamma(1+1/p)^d}{\pi^{d/2}}\E_X \frac{1}{\sqrt{Y_n}},
\end{align}
our goal is thus to obtain a limit theorem for the variable
\begin{align}\label{eq:goal}
  \E_X \frac{1}{\sqrt{Y_n}} - \frac{\pi^{d/2}}{2^{d}\Gamma(1+1/p)^d} \Big(a_{p,d} - \frac{1}{n}b_{p,d}\Big).
\end{align}

Note that in the above expression the random variables $X_i$ are integrated out and it is a function of the random vectors $G_1,\ldots,G_d$.

\subsection{Linearization}

We will make use of the identity
\begin{multline*}
\frac{1}{\sqrt{y}} = \frac{1}{\sqrt{\mu}} - \frac{y-\mu }{2 \mu ^{3/2}} + \frac{3 (y-\mu )^2}{8 \mu ^{5/2}} - \frac{5 (y-\mu )^3}{16 \mu ^{7/2}} +
\frac{\left(\sqrt{y}-\sqrt{\mu }\right)^4 \left(16 \mu ^{3/2}+5 y^{3/2}+20 \sqrt{\mu } y+29 \mu  \sqrt{y}\right)}{16 \mu ^{7/2} \sqrt{y}},
\end{multline*}
valid for $y, \mu > 0$.

Applying it to $y = Y_n = \sum_{\ii\in\rng} a_\ii X_{i_1}\cdots X_{i_d}$ and
\begin{align}\label{eq:definition-of-mu}
\mu = \E_X Y_n = (\E X_1)^d
\end{align}
(where we used \eqref{eq:sum-of-a's}) and integrating with respect to $X_i$'s we obtain

\begin{multline}\label{eq:three-terms}
\E_X \frac{1}{\sqrt{Y_n}} - \frac{1}{\sqrt{\mu}}\\
= \frac{3\E_X (Y_n - \mu)^2}{8\mu^{5/2}} - \frac{5\E_X (Y_n-\mu)^3}{16 \mu^{7/2}}
+ \E_X \frac{\left(\sqrt{Y_n}-\sqrt{\mu }\right)^4 \left(16 \mu^{3/2}+5 Y_n^{3/2}+20 \sqrt{\mu} Y_n+29 \mu\sqrt{Y_n}\right)}{16 \mu^{7/2} \sqrt{Y_n}}\\
=: I_n + II_n + III_n.
\end{multline}
We will first show that $II_n$ and $III_n$ are $\om(n^{-3/2})$ as $n \to \infty$.

\vskip0.5cm

Let us start with $II_n$. To simplify the notation, for $\ii \in \rng$ denote
\begin{align}\label{eq:fat-X-def}
\mathbf{X}_{\ii} = X_{i_1}\cdots X_{i_d} - \E_X X_{i_1}\cdots X_{i_d}.
\end{align}
For two multi-indices $\ii,\jj \in \rng$ let
\begin{align}\label{eq:definition-of-level-set}
\mathcal{E}(\ii,\jj) = \{(k,l)\in [d]^2 \colon i_k = j_l\}.
\end{align}
For $I,J,K \subset [d]^2$ let
\begin{displaymath}
A(I,J,K) = \{(\ii,\jj,\kk)\in (\rng)^3\colon  \mathcal{E}(\ii,\jj) = I, \mathcal{E}(\jj,\kk) = J, \mathcal{E}(\ii,\kk) = K\}.
\end{displaymath}
Note that if at least two of the sets $I,J,K$ are empty then by independence of $X_i$'s, we have $\E_X \mathbf{X}_{\ii} \mathbf{X}_{\jj} \mathbf{X}_{\kk} = 0$ for all $(\ii,\jj,\kk) \in A(I,J,K)$.

Thus
\begin{align*}
  &\E_X (Y_n - \mu)^3
  =
  \sum_{I,J,K} \sum_{(\ii,\jj,\kk) \in A(I,J,K)} a_\ii a_\jj a_\kk \E_X \mathbf{X}_{\ii} \mathbf{X}_{\jj} \mathbf{X}_{\kk},
\end{align*}
where the outer summation is over all triples $I, J, K$ of subsets of $[d]^2$ in which at least two sets are non-empty. To prove that $II_n = \om(n^{-3/2})$ it thus suffices to show that for each such triple
\begin{align}\label{eq:II_n-summand}
  \sum_{(\ii,\jj,\kk) \in A(I,J,K)}  a_\ii a_\jj a_\kk
  \E_X |\mathbf{X}_{\ii} \mathbf{X}_{\jj} \mathbf{X}_{\kk}| = \om(n^{-3/2}).
\end{align}

Observe first that $|A(I,J,K)| \le n^{3d-2}$. Moreover, by \eqref{eq:2nd-l-infinity-bound-moments} and H\"older's inequality
\begin{displaymath}
a_\ii a_\jj a_\kk = \Om( (\log n)^{3d} n^{-3d}).
\end{displaymath}
Since $X_i's$ have moments of all orders, another application of H\"older's inequality gives $\E_X |\mathbf{X}_{\ii} \mathbf{X}_{\jj} \mathbf{X}_{\kk}| \le C_{d,p}$. These three observations imply  that for any $q \ge 1$ the left hand side of \eqref{eq:II_n-summand} is bounded in $L_q$ by $C_{d,p,q} (\log n)^{3d} n^{-2} = o(n^{-3/2})$. Thus indeed $II_n = \om(n^{-3/2})$.
\medskip

We will now pass to $III_n$. Recall \eqref{eq:negative-Holder-CLT} and \eqref{eq:positive-Holder-CLT}. Let $\kappa = 4\rho/(2\rho - 1)>1$. Using H\"older's inequality applied with exponents
$\kappa,\kappa,2\rho$ and taking into account that $\mu = (\E X_1)^d$ depends only on $p$ and $d$ we obtain that
\begin{multline*}
III_n =  \E_X \frac{\Big(\sqrt{Y_n}-\sqrt{\mu }\Big)^4 \Big(16 \mu^{3/2}+5 Y_n^{3/2}+20 \sqrt{\mu} Y_n+29 \mu\sqrt{Y_n}\Big)}{16 \mu^{7/2} \sqrt{Y_n}}\\
\le C_{p,d} (\E_X |\sqrt{Y_n} -\sqrt{\mu}|^{4\kappa})^{1/\kappa}\Big(\E_X \Big(16 \mu^{3/2}+5 Y_n^{3/2}+20 \sqrt{\mu} Y_n+29 \mu\sqrt{Y_n}\Big)^\kappa\Big)^{1/\kappa}(\E_X Y_n^{-\rho})^{1/(2\rho)}.
\end{multline*}
Using \eqref{eq:positive-Holder-CLT} and \eqref{eq:definition-of-mu}, we see that the second factor on the right hand side above is bounded from above by $C_{p,d}$, whereas an application of \eqref{eq:negative-Holder-CLT} shows that the third factor is bounded by $C_{p,d}$, thus to prove that $III_n = \om(n^{-3/2})$ it is enough to show that
\begin{align}\label{eq:enough-for-IIIn}
  (\E_X |\sqrt{Y_n} -\sqrt{\mu}|^{4\kappa})^{1/\kappa} = \om(n^{-3/2}).
\end{align}

Using the inequality $|\sqrt{y} - \sqrt{\mu}| \le \mu^{-1/2}|y - \mu|$, we get
\begin{displaymath}
  (\E_X |\sqrt{Y_n} -\sqrt{\mu}|^{4\kappa})^{1/\kappa} \le C_{p,d} (\E_X|Y_n - \mu|^{4 \kappa})^{1/\kappa} \le C_{p,d} (\E_X (Y_n - \mu)^2)^2,
\end{displaymath}
where in the last inequality we used \eqref{eq:positive-Holder-CLT}.

Thus it is enough to show that $\E_X(Y_n-\mu)^2 = \om(n^{-3/4})$. Note that $\E_X(Y_n-\mu)^2$ is proportional to the term $I_n$ in \eqref{eq:three-terms}. In the next step of the proof we will perform a careful analysis of this term, from which it will in particular follow that  it is $\Om(n^{-1})$. For transparency let us however provide now a simple argument showing that it is indeed $\om(n^{-3/4})$. Some of the calculations will be also used in the said more precise analysis.

Recall the notation \eqref{eq:definition-of-level-set}. Similarly as for $II_n$, for $I\subset [d]^2$, define $A(I) = \{(\ii,\jj)\in (\rng)^2\colon \mathcal{E}(\ii,\jj) = I\}$. By independence of $X_i$'s we  have

\begin{align}\label{eq:I_n-decomposition}
 \E_X(Y_n - \mu)^2 = \sum_{\emptyset \neq I \subset [d]^2} \sum_{(\ii,\jj)\in A(I)} a_\ii a_\jj \E_X(\mathbf{X}_\ii\mathbf{X}_\jj)
\end{align}

For fixed $I$ and $q \ge 1$, by the triangle inequality in $L_q$ we obtain that as $n\to \infty$,
\begin{align}\label{eq:small-I}
  \Big\|\sum_{\stackrel{\ii,\jj \in \rng}{\mathcal{E}(\ii,\jj) = I}} a_\ii a_\jj \E_X(\mathbf{X}_\ii\mathbf{X}_\jj)\Big\|_q \le C_{p,d}|A(I)| \Big\| \max_{\ii \in \rng} a_\ii^2 \Big\|_q \le  C_{p,d,q} |A(I)|\frac{(\log n)^{2d}}{n^{2d}},
\end{align}
where we again used the fact that $X$ has all moments and \eqref{eq:2nd-l-infinity-bound-moments}.

For $I\neq \emptyset$ we have $|A(I)| \le n^{2d-1}$ and so the right hand side above is bounded by $C_{p,d,q} (\log n)^{2d}n^{-1} = o(n^{-3/4})$.

Summing over $I$ we get the claimed estimate.

\subsection{Analysis of randomly normalized $U$-statistics}

We will now analyze the term $I_n =\frac{3\E_X (Y_n - \mu)^2}{8\mu^{5/2}}$. We will use the notation introduced at the end of the previous section. Note that if $I\subset [d]^2$ satisfies $|I|\ge 2$, then $|A(I)| \le n^{2d-2}$. Thus by \eqref{eq:I_n-decomposition} and \eqref{eq:small-I}, we see that
\begin{displaymath}
  \E_X(Y_n -\mu)^2 = \sum_{I\subset [d]^2, |I|=1} \sum_{\stackrel{\ii,\jj \in \rng}{\mathcal{E}(\ii,\jj) = I}} a_\ii a_\jj \E_X(\mathbf{X}_\ii\mathbf{X}_\jj) + \om(n^{-3/2}).
\end{displaymath}
Since $X_i$'s are i.i.d, for all $\ii,\jj$, such that $|\mathcal{E}(\ii,\jj)| = 1$, we have
\begin{align}\label{Eq:definition-of-nu}
  \E \mathbf{X}_\ii\mathbf{X}_\jj = \E X_1^2 \prod_{k=2}^{2d-1}X_k - (\E X_1)^{2d} = (\E X_1)^{2d-2} \Var(X_1) =:\nu.
\end{align}
Thus
\begin{align}\label{eq:I_n-representation}
I_n = \frac{3\nu}{8\mu^{5/2}} \sum_{I\subset [d]^2, |I|=1} \sum_{\stackrel{\ii,\jj \in \rng}{\mathcal{E}(\ii,\jj) = I}} a_\ii a_\jj + \om(n^{-3/2}).
\end{align}

Recall that $\Gamma_1,\ldots,\Gamma_n$ are the columns of the matrix with rows $G_1,\ldots,G_d$.
Recall also the definition \eqref{eq:def-a_ii} of the coefficients $a_\ii$, which implies that
\begin{align}\label{eq:sum-representation}
  Z_n:= \sum_{I\subset [d]^2, |I|=1} \sum_{\stackrel{\ii,\jj \in \rng}{\mathcal{E}(\ii,\jj) = I}} a_\ii a_\jj = \frac{S_n}{V_n},
\end{align}
where
\begin{align}\label{eq:definition-V_n}
  V_n = \prod_{\ell=1}^d|G_\ell-P_{\ell-1}G_\ell|^4.
\end{align}
and (recall the notation concerning $U$-statistics from Section \ref{sec:U-statistics})
\begin{align}\label{eq:definition-U_n}
S_n = S_n(h) = \sum_{\ii \in \Rng{2d-1}} h(\Gamma_{i_1},\ldots,\Gamma_{i_{2d-1}})
\end{align}
is a $U$-statistic of order $2d-1$ with kernel with $h \colon (\R^d)^{2d-1} \to \R$ defined as
\begin{align}\label{eq:h-definition}
  h(x_1,\ldots,x_{2d-1}) = \frac{1}{(2d-1)!}\sum_{k=1}^{2d-1} \sum_{J\subset [2d-1]\setminus\{k\}, |J| = d-1} \det(\{x_k\}\cup \{x_\ell\}_{\ell \in J})^2\det(\{x_k\}\cup\{x_\ell\}_{\ell \in [2d-1]\setminus\{k\}\setminus J})^2.
\end{align}
We slightly abuse the notation and treat here the determinant squared as a function of a set rather than a sequence of vectors. Note that $h$ is a symmetric function.

Let us first establish a LLN type behaviour of the denominator on the right hand side of \eqref{eq:sum-representation}. Denote $H_\ell = \spa(G_1,\ldots,G_{\ell-1})$ and note that conditionally on $G_1,\ldots,G_{\ell-1}$,  $P_{\ell-1} G_\ell$ is a standard Gaussian vector on $H_\ell$. In particular, using Theorem \ref{thm:Gaussian-hypercontractivity} we obtain that
\begin{displaymath}
\E |P_{\ell-1} G|^q = \E \E(|P_{\ell-1} G|^q|G_1,\ldots,G_{\ell-1}) \le C_q^q d^{q/2}.
\end{displaymath}

Moreover $H_\ell$ is almost surely of dimension $\ell-1 \le d-1$ and $|G_\ell - P_{\ell-1} G|^2 = |G_\ell|^2 - |P_{\ell-1} G_\ell|^2$. Using the fact that by Chebyshev's inequality and Theorem \ref{thm:Gaussian-hypercontractivity}, $||G_\ell|^2 - n| = \Om(\sqrt{n})$, together with H\"older's inequality, one can see that
\begin{displaymath}
  \prod_{\ell=1}^d|G_\ell-P_{\ell-1}G_\ell|^2 = \prod_{\ell=1}^d|G_\ell|^2 + \Om(n^{d-1}) = n^{d} + \Om(n^{d-1/2}).
\end{displaymath}
As a consequence, again by H\"older's inequality,
\begin{align}\label{eq:denominator-preliminary}
  V_n = \prod_{\ell=1}^d|G_\ell|^4  + \Om(n^{2d-1}) =  n^{2d} + \Om(n^{2d-1/2}).
\end{align}

Observe also that if $G$ is a standard Gaussian vector in $\R^m$ and $m > q+3$, then integrating in polar coordinates together with the Stirling formula gives
\begin{displaymath}
  \E |G|^{-q} = \E |G|^2 \frac{2^{(m-q-2)/2}\Gamma\Big(\frac{m-q}{2}\Big)}{2^{m/2}\Gamma\Big(\frac{m}{2}+1\Big)} \le C_q m^{-q/2}.
\end{displaymath}
Conditionally on $G_1,\ldots,G_{\ell-1}$, the random vector $G_\ell - P_{\ell-1}G_\ell$ is  a standard Gaussian vector on $H_\ell^\perp$, which is of dimension at least $n-d$. Thus, conditioning successively, we obtain that for $n> C_{d,q}$,
\begin{align}\label{eq:V_n-negative-moments}
  \Big\|\frac{n^{2d}}{V_n}\Big\|_q \le C_{d,q}.
\end{align}

Let us now pass to the numerator. To shorten the notation, denote $\E h = \E h(\Gamma_1,\ldots,\Gamma_d)$. By the Hoeffding decomposition \eqref{eq:Hoeffding-Sn}, applied with $2d-1$ instead of $d$, we have
\begin{align}\label{Eq:S_n-Hoeffding}
  S_n = \frac{n!}{(n-2d+1)!} \E h + \sum_{k=1}^{2d-1} \binom{2d-1}{k} \frac{(n-k)!}{(n-2d+1)!} S_n^{(k)}(\pi_k h).
\end{align}
By \eqref{eq:canonical-second-moment} and \eqref{eq:Hoeffding-contraction property} we have
\begin{displaymath}
  \E |S_n^{(k)}(\pi_k h)|^2 = \frac{n!k!}{(n-k)!} \E (\pi_k h(\Gamma_1,\ldots,\Gamma_k))^2 \le \frac{n!k!}{(n-k)!} \E h(\Gamma_1,\ldots,\Gamma_k)^2.
\end{displaymath}

The variables $S_n^{(k)}(\pi_k h)$ are polynomials in the Gaussian vectors $G_1,\ldots,G_d$ of degree at most $d^{4}$. Thus the above estimate in combination with Theorem \ref{thm:Gaussian-hypercontractivity} shows that
\begin{displaymath}
  \frac{(n-k)!}{(n-2d+1)!} S_n^{(k)}(\pi_k h) = \Om(n^{2d-1-k/2}).
\end{displaymath}
H\"older's inequality and \eqref{eq:V_n-negative-moments} give thus
\begin{displaymath}
  \frac{(n-k)!}{(n-2d+1)!} \frac{S_n^{(k)}(\pi_k h)}{V_n} = \Om(n^{-1-k/2}).
\end{displaymath}
Combined with \eqref{eq:sum-representation} and \eqref{Eq:S_n-Hoeffding}, this shows that
\begin{multline}\label{eq:Z_n-decomposition-1}
  Z_n = \frac{n!}{(n-2d+1)!} \frac{ \E h}{V_n} + (2d-1) \frac{(n-1)!}{(n-2d+1)!} \frac{S_n^{(1)}(\pi_1 h)}{V_n} + \Om(n^{-2})\\
  = n^{2d-1} \frac{ \E h}{V_n} + (2d-1) \frac{(n-1)!}{(n-2d+1)!} \frac{S_n^{(1)}(\pi_1 h)}{V_n} + \Om(n^{-2}),
\end{multline}
where in the second equality we again used \eqref{eq:V_n-negative-moments}.

Let us now derive the announced more precise asymptotics for the denominator. In what follows we will repeatedly and without a direct reference use Theorem \ref{thm:Gaussian-hypercontractivity}, to pass from boundedness of absolute moments of some fixed order for a sequence of Gaussian polynomials to the assertion that this sequence is in fact $\Om(1)$.

Using \eqref{eq:denominator-preliminary} we obtain
\begin{align*}
  V_n  &= \prod_{\ell=1}^d|G_\ell|^4 + \Om(n^{2d-1})\\
  & = \prod_{\ell=1}^d|\sum_{j=1}^n g_{\ell,j}^2|^2 + \Om(n^{2d-1})\\
  &= \sum_{\ii \in \Rng{2d}} \prod_{\ell=1}^d g_{\ell,i_{2\ell-1}}^2g_{\ell,i_{2\ell}}^2 + \sum_{\ii \in [n]^{2d}\setminus \Rng{2d}} \prod_{\ell=1}^d g_{\ell,i_{2\ell-1}}^2g_{\ell,i_{2\ell}}^2
+ \Om(n^{2d-1}).
\end{align*}
Using boundedness of moments of $g_{\ell,i}$'s together with H\"older's inequality and the fact that $|[n]^{2d}\setminus \Rng{2d}| \le C_d n^{2d-1}$ one obtains that the first moment of the second summand on the right hand side above is $\mathcal{O}(n^{2d-1})$. Thus
\begin{align*}
  V_n &= \sum_{\ii \in \Rng{2d}} \prod_{\ell=1}^d g_{\ell,i_{2\ell-1}}^2g_{\ell,i_{2\ell}}^2 + \Om(n^{2d-1})= \sum_{\ii \in \Rng{2d}} \prod_{j=1}^{2d} g_{\lfloor (j+1)/2\rfloor,i_j}^2 + \Om(n^{2d-1})\\
  &= \sum_{I\subseteq [2d]} \frac{(n-|I|)!}{(n-2d)!} \sum_{\ii_I \in \Rng{I}} \prod_{j\in I} (g_{\lfloor (j+1)/2\rfloor,i_j}^2-1) + \Om(n^{2d-1}).
\end{align*}
For $I\neq \emptyset$, using independence of  $g_{\ell,i}$'s and the equality $\E(g_{\ell,i}^2-1)^2 = 2$, we get
\begin{displaymath}
  \E \Big(\sum_{\ii_I \in \Rng{I}} \prod_{j\in I} (g_{\lfloor (j+1)/2\rfloor,i_j}^2-1)\Big)^2 = 2^{|I|} \frac{n!|I|!}{(n-|I|)!},
\end{displaymath}
thus
\begin{displaymath}
  \frac{(n-|I|)!}{(n-2d)!} \sum_{\ii_I \in \Rng{I}} \prod_{j\in I} (g_{\lfloor (j+1)/2\rfloor,i_j}^2-1) = \Om(n^{2d-|I|/2}).
\end{displaymath}
For $|I|\ge 2$ this is $\Om(n^{2d-1})$ and so we obtain that
\begin{align*}
  V_n &= \frac{n!}{(n-2d)!} + 2\frac{(n-1)!}{(n-2d)!}\sum_{\ell=1}^d \sum_{i=1}^n (g_{\ell,i}^2 -1) + \Om(n^{2d-1})\\
  & = n^{2d} + 2\frac{(n-1)!}{(n-2d)!} \sum_{\ell=1}^d \sum_{i=1}^n (g_{\ell,i}^2 -1) + \Om(n^{2d-1})\\
  & = n^{2d} + 2\frac{(n-1)!}{(n-2d)!} \sum_{i=1}^n (|\Gamma_i|^2 -d) + \Om(n^{2d-1}).
\end{align*}

Going back to \eqref{eq:Z_n-decomposition-1} we can write
\begin{align*}
  Z_n - \frac{\E h}{n} & = (\E h) \frac{n^{2d} - V_n}{n V_n} + (2d-1)\frac{(n-1)!}{(n-2d+1)!}\frac{S_n^{(1)}(\pi_1 h)}{V_n} + \Om(n^{-2})\\
  &= \frac{n^{2d}}{V_n} \frac{(n-1)!}{(n-2d)!n^{2d}} \Big(-2(\E h) \frac{1}{n} \sum_{i=1}^n(|\Gamma_i|^2 -d) + (2d-1)\frac{1}{n-2d+1} \sum_{i=1}^n \pi_1 h(\Gamma_i)\Big)  + \Om(n^{-2}),
\end{align*}
where in the second equality we used once more \eqref{eq:V_n-negative-moments} and H\"older's inequality.

Taking into account that $V_n/n^{2d}$ converges in probability to one and that the random vector
\begin{displaymath}
  \Big(\frac{1}{\sqrt{n}} \sum_{i=1}^n(|\Gamma_i|^2 -d), \frac{1}{\sqrt{n}} \sum_{i=1}^n \pi_1 h(\Gamma_i)\Big)
\end{displaymath}
converges weakly  to a Gaussian vector with covariance matrix
\begin{displaymath}
  \left(\begin{array}{cc}
  2d & \Cov(|\Gamma_i|^2 -d),\pi_1 h(\Gamma_1))\\
  \Cov(|\Gamma_i|^2 -d),\pi_1 h(\Gamma_1)) & \E (\pi_1 h(\Gamma_1))^2
  \end{array}\right),
\end{displaymath}
we obtain that
\begin{displaymath}
  n^{3/2}\Big(Z_n - \frac{\E h}{n}\Big)
\end{displaymath}
converges weakly  to a mean zero Gaussian variable with variance
\begin{displaymath}
  8(\E h)^2 d + (2d-1)^2\E(\pi_1 h(\Gamma_1))^2 - 4(2d-1)\E h\cdot \Cov(|\Gamma_1|^2-d,\pi_1 h(\Gamma_1)).
\end{displaymath}

Moreover, another application of \eqref{eq:V_n-negative-moments} together with Theorem \ref{thm:Gaussian-hypercontractivity}
and H\"older's inequality shows that for each $q > 0$ the $q$-th absolute moment of $n^{3/2}\Big(Z_n - \frac{\E h}{n}\Big)$ is bounded independently of $n$, which shows that the convergence in fact holds in $\Wass_q$ for any $q > 0$ (see the remark before Theorem \ref{thm:CLT}).

Using \eqref{eq:I_n-representation} and \eqref{eq:sum-representation} together with \eqref{eq:three-terms} and the fact that $II_n$ and $III_n$ are $o_m(n^{-3/2})$ we obtain that
\begin{displaymath}
n^{3/2}\Big(\E\frac{1}{\sqrt{Y_n}} - \frac{1}{\sqrt{\mu}} - \frac{3\nu \E h}{8\mu^{5/2} n} \Big)  = n^{3/2}\Big(I_n - \frac{3\nu \E h}{8\mu^{5/2}n}\Big) +\om(1)
\end{displaymath}
converges in $\Wass_q$ to a centered Gaussian variable with variance
\begin{align}\label{eq:sigma}
  \widetilde{\Sigma}_{p,d}^2 = \frac{9\nu^2}{64 \mu^5} \Big(8 (\E h)^2 d + (2d-1)^2\E(\pi_1 h(\Gamma_1))^2 - 4(2d-1)\E h\cdot \Cov(|\Gamma_1|^2-d,\pi_1 h(\Gamma_1))\Big).
\end{align}

\subsection{Calculation of the parameters}

In order to obtain an explicit formula for $\Sigma_{p,d}^2$ we need to calculate
\begin{displaymath}
  \mu,\; \nu, \; \E h(\Gamma_1),\;  \E (\pi_1 h(\Gamma_1))^2, \; \Cov(|\Gamma_1|^2-d,\pi_1 h(\Gamma_1)).
\end{displaymath}

The calculation of the first two parameters is straightforward as they are expressed in terms of moments of the random variables $X_1$, which are known thanks to \eqref{eq:stable-moments}. The remaining parameters are moments of Gaussian polynomials of fixed degree and their calculation involves additional combinatorial arguments.

Recalling that $\mu = \E Y_n = (\E X_1)^d$ where $X_1 = W_1^{-1}$ and $W_1$ has density $x^{-1/2}g_{p/2}(x)$ and using Lemma \ref{le:W-moments} with $\alpha = p/2$ and $q = -1$ we obtain
\begin{align}\label{eq:mu-calculation}
\mu = (\E W_1^{-1})^d = \Big(\frac{\Gamma\Big({3}/{p}\Big)\Gamma\Big({1}/{2}\Big)}{\Gamma\Big({3}/{2}\Big)\Gamma\Big({1}/{p}\Big)}\Big)^d =
\Big(2\frac{\Gamma\Big({3}/{p}\Big)}{\Gamma\Big({1}/{p}\Big)}\Big)^d,
\end{align}

Similarly,
\begin{multline}\label{eq:nu-calculation}
\nu = (\E X_1)^{2d-2} \Var(X_1) = (\E W_1^{-1})^{2d-2}(\E W_1^{-2} - (\E W_1^{-1})^2)\\
  = \Big(\frac{\Gamma\Big({3}/{p}\Big)\Gamma\Big({1}/{2}\Big)}{\Gamma\Big({3}/{2}\Big)\Gamma\Big({1}/{p}\Big)}\Big)^{2d-2}\Big(\frac{\Gamma\Big({5}/{p}\Big)\Gamma\Big({1}/{2}\Big)}{\Gamma\Big({5}/{2}\Big)\Gamma\Big({1}/{p}\Big)} - \Big(\frac{\Gamma\Big({3}/{p}\Big)\Gamma\Big({1}/{2}\Big)}{\Gamma\Big({3}/{2}\Big)\Gamma\Big({1}/{p}\Big)}\Big)^{2}\Big)\\
  = \Big(2\frac{\Gamma\Big({3}/{p}\Big)}{\Gamma\Big({1}/{p}\Big)}\Big)^{2d-2}\Big(\frac{4\Gamma\Big({5}/{p}\Big)}{3\Gamma\Big({1}/{p}\Big)} - \Big(2\frac{\Gamma\Big({3}/{p}\Big)}{\Gamma\Big({1}/{p}\Big)}\Big)^{2}\Big).
\end{multline}

Let us now pass to the calculation of $\E h$ and $\pi_1 h(\Gamma_1)$. Recall that
\begin{displaymath}
\pi_1 h(\Gamma_1) = \E_{\Gamma_2,\ldots,\Gamma_{2d-1}} h(\Gamma_1,\ldots,\Gamma_{2d-1}) - \E h.
\end{displaymath}
 We will calculate the first summand and then integrate it to get the other one. To simplify the notation, let us denote $\E' := \E_{\Gamma_2,\ldots,\Gamma_{2d-1}}$

Recalling the definition of $h$, given in \eqref{eq:h-definition} one can see that
\begin{displaymath}
  \E' h(\Gamma_1,\ldots,\Gamma_d) = \frac{1}{(2d-1)!}\binom{2d-2}{d-1}  ( (2d-2)D_1 + D_2),
\end{displaymath}
where
\begin{align*}
D_1 &= \E' \det(\{\Gamma_1,\ldots,\Gamma_d\})^2\det(\{\Gamma_2,\Gamma_{d+1},\ldots,\Gamma_{2d-1}\})^2,\\
 D_2 &= \E' \det(\{\Gamma_1,\ldots,\Gamma_d\})^2\det(\{\Gamma_1,\Gamma_{d+1},\ldots,\Gamma_{2d-1}\})^2.
\end{align*}
Let $Q_i$, $i=0,\ldots,d-1$, be the orthogonal projection onto $\spa(\Gamma_1,\ldots,\Gamma_i)^\perp \subset \R^d$, $Q_1'$ be the orthogonal projection onto $\spa(\Gamma_2)^\perp$ and $Q_i'$, $i=2,\ldots,d$ be the orthogonal projection onto $\spa(\Gamma_2,\Gamma_{d+1},\ldots,\Gamma_{d+i-1})^\perp$. Then, using the interpetation of the determinant as the volume of the paralellopiped we get
\begin{displaymath}
  D_1 = |\Gamma_1|^2 \E' |Q_1 \Gamma_2|^2\cdots |Q_{d-1}\Gamma_d|^2 \cdot |\Gamma_2|^2 |Q_1' \Gamma_{d+1}|^2\cdot |Q_2' \Gamma_{d+2}|^2 \cdots |Q_{d-1}' \Gamma_{2d-1}|^2.
\end{displaymath}
Since conditionally on $\Gamma_2,\ldots,\Gamma_i$, $Q_i \Gamma_{i+1}$ is a standard Gaussian vector on a certain subspace of dimension $d-i$, and an analogous property holds for $Q_i' \Gamma_{d+i}$, we have
\begin{displaymath}
  D_1 = |\Gamma_1|^2 \E' |Q_1 \Gamma_2|^2 |\Gamma_2|^2 (d-2)! (d-1)! = (d-2)! (d-1)! |\Gamma_1|^2 \E' (|Q_1 \Gamma_2|^4  + |Q_1 \Gamma_2|^2 (|\Gamma_2 - Q_1 \Gamma_2|^2).
\end{displaymath}
Using the fact that conditionally on $\Gamma_1$, the random vectors $Q_1 \Gamma_2$ and $\Gamma_2 - Q_1 \Gamma_2$ are independent and have standard Gaussian distributions on spaces of dimension $d-1$ and $1$ respectively, we obtain
\begin{displaymath}
  D_1 = (d-2)! (d-1)! |\Gamma_1|^2 \E (\sum_{i=1}^{d-1} g_i^2)^2 + \E (\sum_{i=1}^{d-1} g_i^2)g_d^2,
\end{displaymath}
where $g_i$'s are i.i.d. $\mathcal{N}(0,1)$ variables. Thus
\begin{displaymath}
  D_1 = (d-2)! (d-1)! |\Gamma_1|^2  ( 3(d-1) + (d-1)(d-2) +d-1) = (d-1)!^2 (d+2)|\Gamma_1|^2.
\end{displaymath}

Similarly (the calculations are simpler),
\begin{displaymath}
  D_2 = |\Gamma_1|^4 (d-1)!^2,
\end{displaymath}
so we get
\begin{align*}
  \E' h(\Gamma_1) &= \frac{1}{(2d-1)!}\binom{2d-2}{d-1}((2d-2)(d-1)!^2 (d+2)|\Gamma_1|^2 + (d-1)!^2|\Gamma_1|^4) \\
  &= \frac{1}{2d-1}((2d-2)(d+2)|\Gamma_1|^2 + |\Gamma_1|^4).
\end{align*}
Integrating $\E' h(\Gamma_1)$ we get
\begin{align}\label{eq:Eh}
\E h = \frac{1}{2d-1}((2d-2)(d+2)d + 3d + d(d-1)) = (d+2)d.
\end{align}
and finally
\begin{align}\label{eq:pi_1h-calculation}
\pi_1 h(\Gamma_1) = \frac{1}{2d-1}\Big((2d-2)(d+2)(|\Gamma_1|^2 -d) + (|\Gamma_1|^4-3d-d(d-1))\Big).
\end{align}

To calculate the variance of $\pi_1 h(\Gamma_1)$ and its covariance with $|\Gamma_1|^2 - d$ it will be convenient to express it in terms of Hermite polynomials of the variables $g_{\ell,1}$. To simplify the notation let us denote from now on $g_\ell = g_{\ell,1}$.
We have
\begin{multline*}
  \pi_1 h(\Gamma_1) = \frac{1}{2d-1}\Big(2d(d+2)\sum_{\ell=1}^d (g_\ell^2 - 1) + \sum_{\ell =1}^d (g_\ell^4 - 6 g_\ell^2 + 3)  + \sum_{1\le i\neq j \le d} (g_i^2-1)(g_j^2-1) \Big)
\end{multline*}
Taking into account that the summands above are uncorrelated and that the variance of the $k$-th Hermite polynomial equals $k!$, we get
\begin{align}\label{eq:variance-calculation}
  \Var(\pi_1 h(\Gamma_1)) = \frac{8d^3(d+2)^2  + 24 d + 4 d(d-1)}{(2d-1)^2}
\end{align}
and
\begin{align}\label{eq:covariance-calculation}
\Cov(\pi_1 h(\Gamma_1),|\Gamma_1|^2 - d) = \frac{4d^2(d+2)}{2d-1}.
\end{align}
Combining \eqref{eq:sigma}--\eqref{eq:covariance-calculation} we finally obtain
\begin{align*}
  & \widetilde{\Sigma}_{p,d}^2 = \frac{9\nu^2}{64 \mu^5 } \Big(8(\E h)^2 d + (2d-1)^2\E(\pi_1 h(\Gamma_1))^2 - 4(2d-1)\E h\cdot \Cov(|\Gamma_1|^2-d,\pi_1 h(\Gamma_1))\Big)\\
= & \frac{9\Big(\frac{4\Gamma(5/p)}{3\Gamma(1/p)} - 4\Big(\frac{\Gamma(3/p)}{\Gamma(1/p)})^{2}\Big)^2}{64 \Big(2\frac{\Gamma(3/p)}{\Gamma(1/p)}\Big)^{d+4}}\times\\
& \times \Big( 8 (d+2)^2d^2\cdot d
+ (2d-1)^2 \frac{8d^3(d+2)^2  + 24 d + 4 d(d-1)}{(2d-1)^2}\\
&- 4 (2d-1)(d+2)d\cdot \frac{4d^2(d+2)}{2d-1}\Big)\\
=& \frac{\Big(\frac{ \Gamma(5/p)}{\Gamma(1/p)} - 3\Big(\frac{\Gamma(3/p)}{\Gamma(1/p)})^{2}\Big)^2}{2^{d+6} \Big(\frac{\Gamma(3/p)}{\Gamma(1/p)}\Big)^{d+4} } \Big(8(d+2)^2d^3 + 8d^3(d+2)^2  + 24 d + 4 d(d-1) - 16(d+2)^2d^3\Big)\\
=& \frac{\Big(\frac{ \Gamma(5/p)}{\Gamma(1/p)} - 3\Big(\frac{\Gamma(3/p)}{\Gamma(1/p)})^{2}\Big)^2}{2^{d+4} \Big(\frac{\Gamma(3/p)}{\Gamma(1/p)}\Big)^{d+4} } d(d+5).
\end{align*}
Recalling that
\begin{displaymath}
n^{3/2}\Big(\E\frac{1}{\sqrt{Y_n}} - \frac{1}{\sqrt{\mu}} - \frac{3\nu \E h}{8\mu^{5/2} n} \Big)
\end{displaymath}
converges to a centered Gaussian variable with variance $\widetilde{\Sigma}_{p,d}^2$ and going back to \eqref{eq:volume-representation-proof-of-CLT} and \eqref{eq:goal} allows to conclude the proof.

\section{Proof of Theorem \ref{thm:d=1-p-arbitrary}}\label{sec:proof-Edgeworth}

The proof of Theorem \ref{thm:d=1-p-arbitrary} will be based on the volume formula of Theorem \ref{thm:2nd-volume-formula} and the Edgeworth expansion given in Theorem \ref{thm:Edgeworth}.

\begin{proof}[Proof of Theorem \ref{thm:d=1-p-arbitrary}]
Let $g_1,g_2,\ldots$ be a sequence of i.i.d. standard Gaussian random variables and let $u_n = \frac{(g_1,\ldots,g_n)}{\sqrt{g_1^2+\cdot+g_n^2}}$. Then $H = H_n = \spa\{u\}^\perp$ is a random Haar distributed subspace of $\R^n$ of codimension one and it is clearly enough to prove the theorem for this choice of $H$.

Let $Y_1,Y_2,\ldots$ be a sequence of independent random variables with density $e^{-\beta_p^p|x|^p}$ where $\beta_p = 2\Gamma(1+1/p)$. According to Theorem \ref{thm:2nd-volume-formula}
\begin{align}\label{eq:reduction-to-density}
 \frac{ \Vol_{n-1}(B_p^n\cap H)}{\Vol_{n-1}(B_p^{n-1})} = f_{g_1,\ldots,g_n}(0),
\end{align}
where for $\alpha_1,\ldots,\alpha_n \in \R$, $f_{\alpha_1,\ldots,\alpha_n}\colon \R\to [0,\infty)$ is the density of the linear combination $\frac{\alpha_1 Y_1+\cdots+\alpha_n Y_n}{\sqrt{\alpha_1^2+\cdots+\alpha_n^2}}$.
To shorten the notation, let us suppress the dependence on the sequence $(g_i)$ and write simply $f_n$ instead of $f$ (this is a slight abuse of notation which however should not lead to misunderstanding). We may assume that the probability space we consider is of the form $(\Omega,\mathcal{F},\p) =
(\Omega_1\times \Omega_2, \mathcal{F}_1\otimes\mathcal{F}_2,\p_1\otimes\p_2)$ and that the variables $g_i$ depend only on the first coordinate while the variables $Y_i$ on the second one. With some abuse of notation we will thus sometime think of $g_i's$ as random variables defined on $\Omega_1$ and $Y_i's$ as random variables defined on $\Omega_2$. Denote also $X_i = g_i Y_i$. We will treat $X_i$'s as random variables on the space $\Omega_2$, for the moment fixing the sequence $g_i$. Let us also denote $G_n = (g_1,\ldots,g_n)$. Thus $f_n$ can be also interpreted as conditional density of $\frac{1}{|G_n|}\sum_{i=1}^n g_i Y_i$ with respect to the $\sigma$-field generated by $g_i$'s.

In what follows we will write $\E_G$ and $\E_Y$ to denote integration with respect to $G$ and $Y$.

The variables $Y_n$ are symmetric so their odd moments and cumulants vanish. Moreover, all moments of $Y_n$ are finite and a simple calculation shows that

\begin{align*}
\E Y_n^2 = \frac{\Gamma(3/p)}{4 \Gamma (1/p)\Gamma(1+1/p)^2}, \E Y_n^4 = \frac{\Gamma(5/p)}{16 \Gamma (1/p)\Gamma(1+1/p)^4}.
\end{align*}
Thus by \eqref{eq:low-cumulants} we have
\begin{align}\label{eq:two-cumulants}
\ka_2(Y_n) = \frac{\Gamma(3/p)}{4 \Gamma (1/p)\Gamma(1+1/p)^2}, \; \ka_4(Y_n) = \frac{1}{16\Gamma(1+1/p)^4}\Big(\frac{\Gamma(5/p)}{\Gamma (1/p)} - 3 \frac{\Gamma(3/p)^2}{\Gamma (1/p)^2}\Big)
\end{align}

We will now apply the Edgeworth expansion given in Theorem \ref{thm:Edgeworth} to $f_n$. Let us verify that the sequence $(X_i)$ satisfies the assumptions of this theorem $\p_G$-almost surely. We will use the notation introduced in the formulation of the theorem.

We have
\begin{displaymath}
B_n = (\E Y_1^2) \sum_{i=1}^n g_i^2
\end{displaymath}
and thus by the Strong  Law of Large Numbers $\liminf_{n\to \infty} \frac{1}{n} B_n = (\E_Y Y_1)^2$, $\p_G$-a.s.
Similarly
\begin{displaymath}
  \frac{1}{n} \sum_{i=1}^n \E_Y |X_i|^K = (\E |Y_1|^K)\frac{1}{n} \sum_{i=1}^n |g_i|^K < \infty
\end{displaymath}
$\p_G$-a.s. for any $K > 0$.
Thus the condition (i) holds with $\p_G$-a.s.

We have $\E_Y |X_i|^K  \ind{|X_i| > n^\tau} \le \E_Y |X_i|^{K+1/\tau}/n = \frac{1}{n} |g_i|^{K+1/\tau} \E |Y_1|^{K+1/\tau}$. Thus
\begin{align*}
  \frac{1}{n} \sum_{i=1}^n \E_Y |X_i|^K\ind{|X_i| > n^\tau} \le (\E |Y_1|^{K+1/\tau}) \frac{1}{n^2} \sum_{i=1}^n |g_i|^{K+1/\tau}.
\end{align*}
Again by the SLLN for every $\tau > 0$, the right hand side converges $\p_G$-a.s. to zero, which shows validity of the condition (ii).

To verify the condition (iii) we will use Lemma \ref{le:Edgeworth-conditions}. Let $h$ be the density of $Y_n$ and note that $h'$ is integrable and so $h$ has finite variation say $V$. Moreover the density of $g_i Y_i$ is of the form $|g_i|^{-1} h(g_i^{-1}\cdot)$ and thus its variation equals $|g_i|^{-1} V \le V$ for $i \in \mathcal{I} = \{j\colon |g_j|> 1\}$. Using one more time the SLLN we see that $\p_G$-a.s.
\begin{displaymath}
  \lim_{n\to \infty} \frac{|\mathcal{I}\cap [n]|}{n} = \p(|g_1|> 1) > 0.
\end{displaymath}
Therefore the condition (iii) of Theorem \ref{thm:Edgeworth} is satisfied $\p_G$-a.s. by Lemma \ref{le:Edgeworth-conditions}.

We have thus proved that the assumptions of Theorem \ref{thm:Edgeworth} hold $\p_G$-a.s. for all $K \ge 3$.  We will however use it for $K = 5$ and only for $x=0$. Thanks to the symmetry of $X_n$ (recall that it implies that odd cumulants vanish) the expansion \eqref{eq:Edgeworth-expansion} will be actually simplified. Recall also that $\ka_m(X_i) = g_i^m \ka_m(Y_i)$. Using the notation of section \ref{sec:Edgeworth} we obtain by \eqref{eq:lambda-definition} that
\begin{displaymath}
  \lambda_{3,n} = \lambda_{5,n} = 0
\end{displaymath}
and
\begin{displaymath}
\lambda_{4,n} = n \frac{\ka_4(Y_1)}{(\E Y_1^2)^2} \frac{\sum_{i=1}^n g_i^4}{(\sum_{i=1}^n g_i^2)^2} .
\end{displaymath}

Combining this with \eqref{eq:terms-of-expansion} and \eqref{eq:first-Hermites} we get
\begin{align*}
  q_{0,n}(x) &= \frac{1}{\sqrt{2\pi}}e^{-x^2/2},\\
  q_{1,n}(x) & = q_{3,n}(x) = 0,\\
  q_{2,n}(x) &= \frac{1}{24\sqrt{2\pi}} e^{-x^2/2}(x^4-6x^2+3) n \frac{\ka_4(Y_1)}{(\E Y_1^2)^2} \frac{\sum_{i=1}^n g_i^4}{(\sum_{i=1}^n g_i^2)^2}.
\end{align*}

Thus $\p_G$-a.s. we have as $n\to \infty$,
\begin{align}\label{eq:density-expansion}
  f_n(0) = \frac{1}{\sqrt{2\pi}(\E Y_1^2)^{1/2}}\Big(1+ \frac{1}{8}\frac{\ka_4(Y_1)}{(\E Y_1^2)^2}\frac{\sum_{i=1}^n g_i^4}{(\sum_{i=1}^n g_i^2)^2}\Big) + o(n^{-3/2}).
\end{align}

Note that the asymptotic behaviour of the random variable $\frac{\sum_{i=1}^n g_i^4}{(\sum_{i=1}^n g_i^2)^2}$ has already been analysed as a special case $d=1$ in the proof of Theorem \ref{thm:CLT}. However, since the elementary and easy analysis is there hidden in the rather involved formalism of general $U$-statistics, let us repeat it here for completeness.

We have
\begin{multline*}
\frac{\sum_{i=1}^n g_i^4}{(\sum_{i=1}^n g_i^2)^2} = \frac{3}{n} \\
+ \frac{1}{n^{3/2}}\Big(
\frac{(n-3)\sum_{i=1}^n (g_i^4 - 6g_i^2 +3)}{n^{3/2}} - \frac{12\sum_{i=1}^n (g_i^2-1) + 3\sum_{1\le i\neq j\le n} (g_i^2-1)(g_j^2-1) + 6n}{n^{3/2}}
\Big)\cdot \frac{n^2}{(\sum_{i=1}^n g_i^2)^2}.
\end{multline*}

The last factor on the right hand side above converges a.s. to 1 by the Law of Large Numbers. The second quotient in parentheses converges in probability to zero as can be easily seen by calculating the variances (note that the summands in the numerator are multiples of Hermite polynomials of different degrees and are thus uncorrelated). The first quotient converges weakly by the CLT to a mean zero Gaussian variable with variance 24.

Using \eqref{eq:density-expansion}, \eqref{eq:reduction-to-density} and \eqref{eq:two-cumulants} we thus obtain that
\begin{multline*}
  n^{3/2}\Big(\frac{\Vol_{n-1}(B_p^n\cap H)}{\Vol_{n-1}(B_p^{n-1})} - a_{p,1} - \frac{1}{n}b_{p,1}\Big)\\
  = n^{3/2}\Big(\frac{\Vol_{n-1}(B_p^n\cap H)}{\Vol_{n-1}(B_p^{n-1})} - \frac{\sqrt{2}\Gamma(1+1/p)\Gamma(1/p)^{1/2}}{\sqrt{\pi} \Gamma(3/p)^{1/2}}\Big(1 + \frac{3}{8n}\Big(\frac{\Gamma(1/p)}{\Gamma(3/p)}\Big)^2\Big(\frac{\Gamma(5/p)}{\Gamma (1/p)} - 3 \frac{\Gamma(3/p)^2}{\Gamma (1/p)^2}\Big)\Big)\Big)\\
  = n^{3/2}\Big(f_n(0) - \frac{1}{\sqrt{2\pi}(\E Y_1^2)^{1/2}}\Big(1+\frac{3}{8n}\frac{\ka_4(Y_1)}{(\E Y_1^2)^2}\Big)\Big) + o_\p(1) \\
  = \frac{\ka_4(Y_1)}{8 \sqrt{2\pi}(\E Y_1^2)^{5/2}}n^{3/2}\Big(\frac{\sum_{i=1}^n g_i^4}{(\sum_{i=1}^n g_i^2)^2} - \frac{3}{n}\Big) + o_\p(1)
  \end{multline*}
converges in distribution to a mean zero Gaussian variable with variance
\begin{align*}
&\frac{ \ka_4(Y_1)^2}{128 \pi (\E Y_1^2)^5}\cdot 24 = \frac{24}{128\pi}\Big(\frac{4 \Gamma(1/p)\Gamma(1+1/p)^2}{\Gamma(3/p)}\Big)^5 \frac{1}{256\Gamma(1+1/p)^8}\Big(\frac{\Gamma(5/p)}{\Gamma (1/p)} - 3 \frac{\Gamma(3/p)^2}{\Gamma (1/p)^2}\Big)^2\\
& = \frac{3}{4\pi}\Big(\frac{\Gamma(1/p)}{\Gamma(3/p)}\Big)^5\Gamma(1+1/p)^2 \Big(\frac{\Gamma(5/p)}{\Gamma (1/p)} - 3 \frac{\Gamma(3/p)^2}{\Gamma (1/p)^2}\Big)^2 = \Sigma_{p,1}^2,
\end{align*}
which ends the proof of Theorem \ref{thm:d=1-p-arbitrary}.
\end{proof}

\begin{remark}
Note that the term $o(n^{-3/2})$ in \eqref{eq:density-expansion} in general depends on the values of the sequence $(g_i)$ and is not given explicitly. For this reason using Theorem \ref{thm:Edgeworth} as a black box will not lead to convergence in Wasserstein distance, contrary to the proof of Theorem \ref{thm:CLT}.
\end{remark}

\begin{remark}\label{rem:Edgeworth-in-high-d}

In principle the method of proof of Theorem \ref{thm:d=1-p-arbitrary} should work for general $d$. What one would need is a suitable version of multidimensional Edgeworth expansion for the density of sums of independent but non-identically distributed random vectors (actually only for the value of density at zero). The majority of the literature on Edgeworth expansions in higher dimensions focuses on sums of i.i.d. variables however there are several results concerning the non i.i.d. setting (see, e.g., the monograph \cite{MR0436272}). One of the main difficulties in applying such theorems as black boxes for $d>1$ is that due to the Gram-Schmidt orthogonalization performed for each $n$ in order to relate a random basis of $H^\perp$ to Gaussian vectors, one actually would need Edgeworth expansions not for infinite sequences of random variables but for triangular arrays. We are not aware of a result of this type for densities which would be easily applicable in our setting. It is quite likely that such a result can be obtained by  an appropriate adaptation of the proofs of known theorems for sequences of random vectors. Such an extension is however beyond the scope of this article.
\end{remark}

\begin{remark}\label{rem:generic-CLT-etc}
Let us note that Edgeworth expansions for randomly weighted sums of independent real-valued random variables have been recently investigated in \cite{MR4175744}. The results obtained therein concern rather approximations of cumulative distribution functions than densities and the average error in the Edgeworth expansion up to order four. While not directly applicable to our setting, they share some similarities, in particular they show that the average approximation error for the Edgeworth expansion with deterministic terms (i.e., independent of the direction) in a typical situation is of the order $n^{-3/2}$ which agrees with the normalization in our limit theorems. The results obtained \cite{MR4175744} complement an earlier work \cite{MR3136463} in which Berry-Esseen bounds for randomly weighted sums were investigated. Related Berry-Esseen bounds for random vectors in higher dimension were also recently investigated in \cite{MR4193896}. Let us mention that this direction of research has actually been initiated already by V.N. Sudakov in the late 1970s \cite{MR517198}.
\end{remark}

\section{Proof of Theorem \ref{thm:cube}}\label{sec:cube}

The proof is similar and simpler than the proof of Theorem \ref{thm:d=1-p-arbitrary} so we will just indicate the necessary modifications.

Considering again an i.i.d. sequence $g_1,g_2\ldots$ of standard Gaussian variables and letting $u_n = \frac{(g_1,\ldots,g_n)}{\sqrt{g_1^2+\ldots+g_n^2}}$, it is now straightforward to see that
\begin{displaymath}
  2^{-n}\Vol_{n-1}(B_\infty^n\cap (xu+H_n)) = f_n(x),
\end{displaymath}
where $f_n$ is the conditional density of $\frac{\sum_{i=1}^n g_i Y_i}{\sqrt{g_1^2+\ldots+g_n^2}}$, where $Y_1,Y_2,\ldots$ is a sequence of i.i.d. random variables, uniform on $[-1,1]$ and independent of the sequence $(g_i)$ (we condition on $(g_i)$).

Thus repeating the steps related to \eqref{eq:density-expansion} we obtain that
\begin{multline*}
  2^{-n}\Vol_{n-1}(B_\infty^n\cap (xu+H_n)) \\
  = \frac{1}{\sqrt{2\pi} (\E Y_1^2)^{1/2}}\exp\Big(-\frac{x^2}{2\E Y_1^2}\Big)\Big(1
  + \frac{1}{24} \Big(\frac{x^4}{(\E Y_1^2)^2} -6\frac{x^2}{\E Y_1^2}+3\Big) \frac{\ka_4(Y_1)}{(\E Y_1^2)^2} \frac{\sum_{i=1}^n g_i^4}{(\sum_{i=1}^n g_i^2)^2}\Big) + o_\p(n^{-3/2}).
\end{multline*}

By the analysis from the proof of Theorem \ref{thm:d=1-p-arbitrary} we thus see that

\begin{multline*}
  2^{-n}\Vol_{n-1}(B_\infty^n\cap (xu+H_n))\\
  = \frac{1}{\sqrt{2\pi} (\E Y_1^2)^{1/2}}\exp\Big(-\frac{x^2}{2\E Y_1^2}\Big)\Big(1
  + \frac{1}{24} \Big(\frac{x^4}{(\E Y_1^2)^2} -6\frac{x^2}{\E Y_1^2}+3\Big) \frac{\ka_4(Y_1)}{(\E Y_1^2)^2}\Big(\frac{3}{n} + U_n\Big)\Big),
\end{multline*}
for a sequence $U_n$ pf random variables such that $n^{3/2}U_n$ converges weakly to a Gaussian variable with mean zero and variance 24.
To finish the proof it is now enough to rearrange the terms and substitute the values of variance and fourth cumulant of the uniform distribution.

\section{Proof of Corollary \ref{cor:intersection}}\label{sec:intersection}
The following argument is a simple application of a delta type method in combination with Theorem \ref{thm:d=1-p-arbitrary}.

\begin{proof}[Proof of Corollary \ref{cor:intersection}] By Theorem \ref{thm:d=1-p-arbitrary} we have
\begin{displaymath}
\frac{\rho_{\mathcal{I}B_p^n}(\eta)}{\Vol_{n-1}(B_p^n)} = a_{p,1} + R_n,
\end{displaymath}
where $R_n$ is a random variable such that $n^{3/2} (R_n - b_{p,1}/n)$ converges in distribution to a centered Gaussian variable with variance $\Sigma_{p,1}^2$. In particular $n R_n$ converges in probability to $b_{p,1}$.

Thus
\begin{align*}
  n^{3/2}\Big(\Vol_{n-1}(B_p^n) \|\eta\|_{\mathcal{I}B_p^n} - \frac{1}{a_{p,1}} + \frac{b_{p,1}}{n a_{p,1}^2}\Big) & = n^{3/2}\Big(\frac{1}{a_{p,1} + R_n}  -\frac{1}{a_{p,1}} + \frac{b_{p,1}}{n a_{p,1}^2}\Big)\\
& = \frac{n^{3/2}R_n^2}{a_{p,1}^2(a_{p,1}+R_n)} - \frac{n^{3/2}}{a_{p,1}^2}\Big(R_n - \frac{b_{p,1}}{n}\Big).
\end{align*}

To finish the proof it suffices to note that the first summand on the right hand side above converges in probability to zero, the second one to $\mathcal{N}(0,\Sigma_{p,1}^2/a_{p,1}^4)$.
\end{proof}

\appendix

\section{Proof of Theorem \ref{thm:N-T}}\label{app:proof-NT}

The argument we present below mimics the proof in \cite{MR4055953} provided there for $p=1$, therefore we will only present a sketch. Our main objective is to derive correct constants on the right hand side of \eqref{eq:N-T}.

\begin{proof}[Proof of Theorem \ref{thm:N-T}]
Let $H(\epsilon) = \{c\in \R^n\colon |\langle x,u_j\rangle| \le \varepsilon/2. \; j=1,\ldots,d\}$. By a well known formula for volumes of sections (see \cite{MR4055953} for a discussion and references)
\begin{align}\label{eq:last-one}
\Gamma(1+ (n-d)/p) \Vol_{n-d} (H\cap B_p^n) &= \lim_{\varepsilon \to 0}\frac{1}{\varepsilon^d}\int_{H(\varepsilon)}e^{-\sum_{j=1}^n |x_i|^p}dx \nonumber \\
& = (2\Gamma(1+1/p))^n\lim_{\varepsilon \to 0} \frac{1}{\varepsilon^d}\p\Big(\Big\|\sum_{i=1}^n X_i v_i\Big\|_\infty\le \varepsilon/2\Big),
\end{align}
where $X_1,\ldots,X_n$ are i.i.d. random variables with density $\frac{1}{2\Gamma(1+1/p)}\exp(-|x|^p)$.
By \cite[Lemma 23]{MR3846841}, $X_i$'s have the same distribution as $(2W_i)^{-1/2}g_i$, where $G= (g_1,\ldots,g_n)$ is a sequence of independent standard Gaussian variables independent of the sequence $(W_i)$.
Thus
\begin{displaymath}
  \p\Big(\Big\|\sum_{i=1}^n X_i v_i\Big\|_\infty\Big) = \p\Big(\Big\|\sum_{i=1}^n g_i \tilde{v}_i\Big\|_\infty\le \varepsilon/2\Big) = \p(G \in (\varepsilon/2)K),
\end{displaymath}
where $\tilde{v}_i = (2W_i)^{-1/2}v_i$, $A$ is the $d\times n$ matrix with columns $\tilde{v}_i$ and $K = A^{-1}B_\infty^d$. Let $V := (\Ker A)^\perp$ and $\tilde{G}$ be the orthogonal projection of $G$ onto $V$. Then $G \in K$ if and only if $\tilde{G} \in K\cap V$. Conditionally on $W_i$'s, $\tilde{G}$ is a standard Gaussian vector on the $d$-dimensional subspace of $V$. It has a continuous and bounded density with respect to the Lebesgue measure on $V$, whose value at zero equals $(2\pi )^{-d/2}$. As a consequence, by the Fubini and Lebesgue dominated convergence theorems,
\begin{align*}
  \lim_{\varepsilon \to 0} \frac{1}{\varepsilon^d }\p\Big(\Big\|\sum_{i=1}^n g_i \tilde{v}_i\Big\|_\infty \le \varepsilon/2\Big)
  &= \E_W \lim_{\varepsilon \to 0} \frac{1}{\varepsilon^d } \p_G(\tilde{G} \in (\varepsilon/2)( K\cap V)) \\
  &= 2^{-d} (2\pi)^{-d/2} \E_W \Vol_d (K\cap V) = 2^{-3d/2} \pi^{-d/2}\E_W \Vol_d (K\cap V)
\end{align*}

It remains to observe that $A$ is a linear isomorphism between $V$ and $\R^d$ and it maps $K\cap V$ onto $B_\infty^d$, which is of volume $2^d$. Thus
\begin{displaymath}
\Vol_d (K\cap V) = (\det (A A^T))^{-1/2} 2^d = 2^{3d/2} \Big(\det\Big(\sum_{j=1}^n \frac{1}{W_j}v_j v_j^T \Big)\Big)^{-1/2},
\end{displaymath}
which combined with the previous formula and \eqref{eq:last-one} gives
\begin{displaymath}
  \Vol_{n-d}(B_p^n \cap H) = \frac{2^n}{\pi^{d/2}}\frac{\Gamma(1+1/p)^n}{\Gamma(1+(n-d)/p)} \E \Big(\det\Big(\sum_{j=1}^n \frac{1}{W_j}v_j v_j^T \Big)\Big)^{-1/2}.
\end{displaymath}
\end{proof}

\medskip

\noindent {\bf Acknowledgements} R.A. would like to thank Katya Blau for support and inspiring interactions.

\bibliographystyle{amsplain}
\bibliography{CLT-sections}

\providecommand{\bysame}{\leavevmode\hbox to3em{\hrulefill}\thinspace}
\providecommand{\MR}{\relax\ifhmode\unskip\space\fi MR }
\providecommand{\MRhref}[2]{%
  \href{http://www.ams.org/mathscinet-getitem?mr=#1}{#2}
}
\providecommand{\href}[2]{#2}
\begin{thebibliography}{10}

\bibitem{MR3052405}
Rados{\l}aw Adamczak and Rafa{\l} Lata{\l}a, \emph{Tail and moment estimates
  for chaoses generated by symmetric random variables with logarithmically
  concave tails}, Ann. Inst. Henri Poincar\'{e} Probab. Stat. \textbf{48}
  (2012), no.~4, 1103--1136. \MR{3052405}

\bibitem{MR4172611}
D.~Alonso-Guti\'{e}rrez, F.~Besau, J.~Grote, Z.~Kabluchko, M.~Reitzner,
  C.~Th\"{a}le, B.-H. Vritsiou, and E.~Werner, \emph{Asymptotic normality for
  random simplices and convex bodies in high dimensions}, Proc. Amer. Math.
  Soc. \textbf{149} (2021), no.~1, 355--367. \MR{4172611}

\bibitem{MR3806754}
David Alonso-Guti\'{e}rrez, Joscha Prochno, and Christoph Th\"{a}le,
  \emph{Large deviations for high-dimensional random projections of
  {$\ell_p^n$}-balls}, Adv. in Appl. Math. \textbf{99} (2018), 1--35.
  \MR{3806754}

\bibitem{MR4003577}
\bysame, \emph{Gaussian fluctuations for high-dimensional random projections of
  {$\ell_p^n$}-balls}, Bernoulli \textbf{25} (2019), no.~4A, 3139--3174.
  \MR{4003577}

\bibitem{MR1997580}
Milla Anttila, Keith Ball, and Irini Perissinaki, \emph{The central limit
  problem for convex bodies}, Trans. Amer. Math. Soc. \textbf{355} (2003),
  no.~12, 4723--4735. \MR{1997580}

\bibitem{MR3331351}
Shiri Artstein-Avidan, Apostolos Giannopoulos, and Vitali~D. Milman,
  \emph{Asymptotic geometric analysis. {P}art {I}}, Mathematical Surveys and
  Monographs, vol. 202, American Mathematical Society, Providence, RI, 2015.
  \MR{3331351}

\bibitem{MR4169169}
Anastas Baci, Carina Betken, Anna Gusakova, and Christoph Th\"{a}le,
  \emph{Concentration inequalities for functionals of {P}oisson cylinder
  processes}, Electron. J. Probab. \textbf{25} (2020), Paper No. 128, 27.
  \MR{4169169}

\bibitem{MR1008726}
Keith Ball, \emph{Volumes of sections of cubes and related problems}, Geometric
  aspects of functional analysis (1987--88), Lecture Notes in Math., vol. 1376,
  Springer, Berlin, 1989, pp.~251--260. \MR{1008726}

\bibitem{MR2330981}
Imre B\'{a}r\'{a}ny and Van Vu, \emph{Central limit theorems for {G}aussian
  polytopes}, Ann. Probab. \textbf{35} (2007), no.~4, 1593--1621. \MR{2330981}

\bibitem{MR2123199}
Franck Barthe, Olivier Gu\'{e}don, Shahar Mendelson, and Assaf Naor, \emph{A
  probabilistic approach to the geometry of the {$l^n_p$}-ball}, Ann. Probab.
  \textbf{33} (2005), no.~2, 480--513. \MR{2123199}

\bibitem{MR0436272}
R.~N. Bhattacharya and R.~Ranga~Rao, \emph{Normal approximation and asymptotic
  expansions}, Wiley Series in Probability and Mathematical Statistics, John
  Wiley \& Sons, New York-London-Sydney, 1976. \MR{0436272}

\bibitem{MR4193896}
S.~G. Bobkov, G.~P. Chistyakov, and F.~G\"{o}tze, \emph{Poincar\'{e}
  inequalities and normal approximation for weighted sums}, Electron. J.
  Probab. \textbf{25} (2020), Paper No. 155, 31. \MR{4193896}

\bibitem{MR4175744}
Sergey~G. Bobkov, \emph{Edgeworth corrections in randomized central limit
  theorems}, Geometric aspects of functional analysis. {V}ol. {I}, Lecture
  Notes in Math., vol. 2256, Springer, Cham, [2020] \copyright 2020,
  pp.~71--97. \MR{4175744}

\bibitem{Bon}
Aline Bonami, \emph{\'{E}tude des coefficients de {F}ourier des fonctions de
  {$L^{p}(G)$}}, Ann. Inst. Fourier (Grenoble) \textbf{20} (1970), no.~fasc. 2,
  335--402 (1971). \MR{MR0283496 (44 \#727)}

\bibitem{MR3185453}
Silouanos Brazitikos, Apostolos Giannopoulos, Petros Valettas, and
  Beatrice-Helen Vritsiou, \emph{Geometry of isotropic convex bodies},
  Mathematical Surveys and Monographs, vol. 196, American Mathematical Society,
  Providence, RI, 2014. \MR{3185453}

\bibitem{MR3405618}
Pierre Calka and J.~E. Yukich, \emph{Variance asymptotics and scaling limits
  for {G}aussian polytopes}, Probab. Theory Related Fields \textbf{163} (2015),
  no.~1-2, 259--301. \MR{3405618}

\bibitem{MR1839474}
Anthony Carbery and James Wright, \emph{Distributional and {$L^q$} norm
  inequalities for polynomials over convex bodies in {$\Bbb R^n$}}, Math. Res.
  Lett. \textbf{8} (2001), no.~3, 233--248. \MR{1839474}

\bibitem{MR4278336}
Debsoumya Chakraborti, Tomasz Tkocz, and Beatrice-Helen Vritsiou, \emph{A note
  on volume thresholds for random polytopes}, Geom. Dedicata \textbf{213}
  (2021), 423--431. \MR{4278336}

\bibitem{chasapis2021slicing}
Giorgos Chasapis, Piotr Nayar, and Tomasz Tkocz, \emph{Slicing $\ell_p$-balls
  reloaded: stability, planar sections in $\ell_1$}, 2021.

\bibitem{MR3368101}
Dario Cordero-Erausquin, Matthieu Fradelizi, Grigoris Paouris, and Peter
  Pivovarov, \emph{Volume of the polar of random sets and shadow systems},
  Math. Ann. \textbf{362} (2015), no.~3-4, 1305--1325. \MR{3368101}

\bibitem{MR1666908}
V\'{\i}ctor~H. de~la Pe\~{n}a and Evarist Gin\'{e}, \emph{Decoupling},
  Probability and its Applications (New York), Springer-Verlag, New York, 1999,
  From dependence to independence, Randomly stopped processes. $U$-statistics
  and processes. Martingales and beyond. \MR{1666908}

\bibitem{MR1139489}
M.~E. Dyer, Z.~F\"{u}redi, and C.~McDiarmid, \emph{Volumes spanned by random
  points in the hypercube}, Random Structures Algorithms \textbf{3} (1992),
  no.~1, 91--106. \MR{1139489}

\bibitem{MR2402109}
R.~Eldan and B.~Klartag, \emph{Pointwise estimates for marginals of convex
  bodies}, J. Funct. Anal. \textbf{254} (2008), no.~8, 2275--2293. \MR{2402109}

\bibitem{MR3846841}
Alexandros Eskenazis, Piotr Nayar, and Tomasz Tkocz, \emph{Gaussian mixtures:
  entropy and geometric inequalities}, Ann. Probab. \textbf{46} (2018), no.~5,
  2908--2945. \MR{3846841}

\bibitem{MR3828740}
\bysame, \emph{Sharp comparison of moments and the log-concave moment problem},
  Adv. Math. \textbf{334} (2018), 389--416. \MR{3828740}

\bibitem{MR4157095}
Alan Frieze, Wesley Pegden, and Tomasz Tkocz, \emph{Random volumes in
  {$d$}-dimensional polytopes}, Discrete Anal. (2020), Paper No. 15, 17.
  \MR{4157095}

\bibitem{MR3737915}
Nina Gantert, Steven~Soojin Kim, and Kavita Ramanan, \emph{Large deviations for
  random projections of {$\ell^p$} balls}, Ann. Probab. \textbf{45} (2017),
  no.~6B, 4419--4476. \MR{3737915}

\bibitem{MR2460902}
D.~Gatzouras and A.~Giannopoulos, \emph{Threshold for the volume spanned by
  random points with independent coordinates}, Israel J. Math. \textbf{169}
  (2009), 125--153. \MR{2460902}

\bibitem{gusakova2022volume}
Anna Gusakova, Johannes Heiny, and Christoph Th{\"a}le, \emph{The volume of
  random simplices from elliptical distributions in high dimension}, arXiv
  preprint arXiv:2206.00514 (2022).

\bibitem{MR4213157}
Anna Gusakova and Christoph Th\"{a}le, \emph{The volume of simplices in
  high-dimensional {P}oisson-{D}elaunay tessellations}, Ann. H. Lebesgue
  \textbf{4} (2021), 121--153. \MR{4213157}

\bibitem{MR15746}
Paul~R. Halmos, \emph{The theory of unbiased estimation}, Ann. Math. Statistics
  \textbf{17} (1946), 34--43. \MR{15746}

\bibitem{MR26294}
Wassily Hoeffding, \emph{A class of statistics with asymptotically normal
  distribution}, Ann. Math. Statistics \textbf{19} (1948), 293--325. \MR{26294}

\bibitem{MR3573332}
Yong Huang, Erwin Lutwak, Deane Yang, and Gaoyong Zhang, \emph{Geometric
  measures in the dual {B}runn-{M}inkowski theory and their associated
  {M}inkowski problems}, Acta Math. \textbf{216} (2016), no.~2, 325--388.
  \MR{3573332}

\bibitem{MR1474726}
Svante Janson, \emph{Gaussian {H}ilbert spaces}, Cambridge Tracts in
  Mathematics, vol. 129, Cambridge University Press, Cambridge, 1997.
  \MR{1474726}

\bibitem{MR3904638}
Zakhar Kabluchko, Joscha Prochno, and Christoph Th\"{a}le,
  \emph{High-dimensional limit theorems for random vectors in
  {$\ell_p^n$}-balls}, Commun. Contemp. Math. \textbf{21} (2019), no.~1,
  1750092, 30. \MR{3904638}

\bibitem{MR4216415}
\bysame, \emph{High-dimensional limit theorems for random vectors in
  {$\ell^n_p$}-balls. {II}}, Commun. Contemp. Math. \textbf{23} (2021), no.~3,
  Paper No. 1950073, 35. \MR{4216415}

\bibitem{MR4260512}
\bysame, \emph{A new look at random projections of the cube and general product
  measures}, Bernoulli \textbf{27} (2021), no.~3, 2117--2138. \MR{4260512}

\bibitem{MR2166308}
N.~J. Kalton and A.~Koldobsky, \emph{Intersection bodies and {$L_p$}-spaces},
  Adv. Math. \textbf{196} (2005), no.~2, 257--275. \MR{2166308}

\bibitem{MR2285748}
B.~Klartag, \emph{A central limit theorem for convex sets}, Invent. Math.
  \textbf{168} (2007), no.~1, 91--131. \MR{2285748}

\bibitem{MR2311626}
\bysame, \emph{Power-law estimates for the central limit theorem for convex
  sets}, J. Funct. Anal. \textbf{245} (2007), no.~1, 284--310. \MR{2311626}

\bibitem{MR3136463}
B.~Klartag and S.~Sodin, \emph{Variations on the {B}erry-{E}sseen theorem},
  Teor. Veroyatn. Primen. \textbf{56} (2011), no.~3, 514--533. \MR{3136463}

\bibitem{MR1796717}
A.~Koldobsky and M.~Lifshits, \emph{Average volume of sections of star bodies},
  Geometric aspects of functional analysis, Lecture Notes in Math., vol. 1745,
  Springer, Berlin, 2000, pp.~119--146. \MR{1796717}

\bibitem{MR1656857}
Alexander Koldobsky, \emph{An application of the {F}ourier transform to
  sections of star bodies}, Israel J. Math. \textbf{106} (1998), 157--164.
  \MR{1656857}

\bibitem{MR2132704}
\bysame, \emph{Fourier analysis in convex geometry}, Mathematical Surveys and
  Monographs, vol. 116, American Mathematical Society, Providence, RI, 2005.
  \MR{2132704}

\bibitem{MR931501}
Wies{\l}aw Krakowiak and Jerzy Szulga, \emph{Hypercontraction principle and
  random multilinear forms}, Probab. Theory Related Fields \textbf{77} (1988),
  no.~3, 325--342. \MR{931501}

\bibitem{MR893914}
Stanis{\l}aw Kwapie\'{n}, \emph{Decoupling inequalities for polynomial chaos},
  Ann. Probab. \textbf{15} (1987), no.~3, 1062--1071. \MR{893914}

\bibitem{MR1085342}
Stanis{\l}aw Kwapie\'{n} and Jerzy Szulga, \emph{Hypercontraction methods in
  moment inequalities for series of independent random variables in normed
  spaces}, Ann. Probab. \textbf{19} (1991), no.~1, 369--379. \MR{1085342}

\bibitem{MR1167198}
Stanis{\l}aw Kwapie\'{n} and Wojbor~A. Woyczy\'{n}ski, \emph{Random series and
  stochastic integrals: single and multiple}, Probability and its Applications,
  Birkh\"{a}user Boston, Inc., Boston, MA, 1992. \MR{1167198}

\bibitem{MR3215537}
G\"{u}nter Last, Mathew~D. Penrose, Matthias Schulte, and Christoph Th\"{a}le,
  \emph{Moments and central limit theorems for some multivariate {P}oisson
  functionals}, Adv. in Appl. Probab. \textbf{46} (2014), no.~2, 348--364.
  \MR{3215537}

\bibitem{MR2449135}
R.~Lata{\l}a and J.~O. Wojtaszczyk, \emph{On the infimum convolution
  inequality}, Studia Math. \textbf{189} (2008), no.~2, 147--187. \MR{2449135}

\bibitem{MR1075417}
A.~J. Lee, \emph{{$U$}-statistics}, Statistics: Textbooks and Monographs, vol.
  110, Marcel Dekker, Inc., New York, 1990, Theory and practice. \MR{1075417}

\bibitem{MR0041346}
Paul L\'{e}vy, \emph{Probl\`emes concrets d'analyse fonctionnelle. {A}vec un
  compl\'{e}ment sur les fonctionnelles analytiques par {F}. {P}ellegrino},
  Gauthier-Villars, Paris, 1951, 2d ed. \MR{0041346}

\bibitem{MR1469422}
A.~E. Litvak, V.~D. Milman, and A.~Pajor, \emph{The covering numbers and ``low
  {$M^\ast$}-estimate'' for quasi-convex bodies}, Proc. Amer. Math. Soc.
  \textbf{127} (1999), no.~5, 1499--1507. \MR{1469422}

\bibitem{MR1645952}
A.~E. Litvak, V.~D. Milman, and G.~Schechtman, \emph{Averages of norms and
  quasi-norms}, Math. Ann. \textbf{312} (1998), no.~1, 95--124. \MR{1645952}

\bibitem{MR380631}
Erwin Lutwak, \emph{Dual mixed volumes}, Pacific J. Math. \textbf{58} (1975),
  no.~2, 531--538. \MR{380631}

\bibitem{MR963487}
\bysame, \emph{Intersection bodies and dual mixed volumes}, Adv. in Math.
  \textbf{71} (1988), no.~2, 232--261. \MR{963487}

\bibitem{MR1890647}
Erwin Lutwak, Deane Yang, and Gaoyong Zhang, \emph{The {C}ramer-{R}ao
  inequality for star bodies}, Duke Math. J. \textbf{112} (2002), no.~1,
  59--81. \MR{1890647}

\bibitem{MR960226}
Mathieu Meyer and Alain Pajor, \emph{Sections of the unit ball of {$L^n_p$}},
  J. Funct. Anal. \textbf{80} (1988), no.~1, 109--123. \MR{960226}

\bibitem{MR0293374}
V.~D. Milman, \emph{A new proof of {A}. {D}voretzky's theorem on cross-sections
  of convex bodies}, Funkcional. Anal. i Prilo\v{z}en. \textbf{5} (1971),
  no.~4, 28--37. \MR{0293374}

\bibitem{MR856576}
Vitali~D. Milman and Gideon Schechtman, \emph{Asymptotic theory of
  finite-dimensional normed spaces}, Lecture Notes in Mathematics, vol. 1200,
  Springer-Verlag, Berlin, 1986, With an appendix by M. Gromov. \MR{856576}

\bibitem{MR3751326}
Ilya Molchanov, \emph{Theory of random sets}, Probability Theory and Stochastic
  Modelling, vol.~87, Springer-Verlag, London, 2017, Second edition of [
  MR2132405]. \MR{3751326}

\bibitem{MR3061780}
Elchanan Mossel, Krzysztof Oleszkiewicz, and Arnab Sen, \emph{On reverse
  hypercontractivity}, Geom. Funct. Anal. \textbf{23} (2013), no.~3,
  1062--1097. \MR{3061780}

\bibitem{MR2262841}
Assaf Naor, \emph{The surface measure and cone measure on the sphere of
  {$l_p^n$}}, Trans. Amer. Math. Soc. \textbf{359} (2007), no.~3, 1045--1079.
  \MR{2262841}

\bibitem{MR1962135}
Assaf Naor and Dan Romik, \emph{Projecting the surface measure of the sphere of
  {$ l_p^n$}}, Ann. Inst. H. Poincar\'{e} Probab. Statist. \textbf{39} (2003),
  no.~2, 241--261. \MR{1962135}

\bibitem{MR4055953}
Piotr Nayar and Tomasz Tkocz, \emph{On a convexity property of sections of the
  cross-polytope}, Proc. Amer. Math. Soc. \textbf{148} (2020), no.~3,
  1271--1278. \MR{4055953}

\bibitem{N}
Edward Nelson, \emph{The free {M}arkoff field}, J. Funct. Anal. \textbf{12}
  (1973), 211--227. \MR{0343816 (49 \#8556)}

\bibitem{MR2921184}
Grigoris Paouris and Peter Pivovarov, \emph{A probabilistic take on
  isoperimetric-type inequalities}, Adv. Math. \textbf{230} (2012), no.~3,
  1402--1422. \MR{2921184}

\bibitem{MR3038532}
\bysame, \emph{Small-ball probabilities for the volume of random convex sets},
  Discrete Comput. Geom. \textbf{49} (2013), no.~3, 601--646. \MR{3038532}

\bibitem{MR3230006}
Grigoris Paouris, Peter Pivovarov, and Joel Zinn, \emph{A central limit theorem
  for projections of the cube}, Probab. Theory Related Fields \textbf{159}
  (2014), no.~3-4, 701--719. \MR{3230006}

\bibitem{MR0388499}
V.~V. Petrov, \emph{Sums of independent random variables}, Ergebnisse der
  Mathematik und ihrer Grenzgebiete, Band 82, Springer-Verlag, New
  York-Heidelberg, 1975, Translated from the Russian by A. A. Brown.
  \MR{0388499}

\bibitem{MR1885651}
Matthias Reitzner, \emph{Random points on the boundary of smooth convex
  bodies}, Trans. Amer. Math. Soc. \textbf{354} (2002), no.~6, 2243--2278.
  \MR{1885651}

\bibitem{MR3161465}
Matthias Reitzner and Matthias Schulte, \emph{Central limit theorems for
  {$U$}-statistics of {P}oisson point processes}, Ann. Probab. \textbf{41}
  (2013), no.~6, 3879--3909. \MR{3161465}

\bibitem{MR1280932}
Gennady Samorodnitsky and Murad~S. Taqqu, \emph{Stable non-{G}aussian random
  processes}, Stochastic Modeling, Chapman \& Hall, New York, 1994, Stochastic
  models with infinite variance. \MR{1280932}

\bibitem{MR1015684}
G.~Schechtman and J.~Zinn, \emph{On the volume of the intersection of two
  {$L^n_p$} balls}, Proc. Amer. Math. Soc. \textbf{110} (1990), no.~1,
  217--224. \MR{1015684}

\bibitem{MR1796723}
\bysame, \emph{Concentration on the {$l^n_p$} ball}, Geometric aspects of
  functional analysis, Lecture Notes in Math., vol. 1745, Springer, Berlin,
  2000, pp.~245--256. \MR{1796723}

\bibitem{MR28600}
Erhard Schmidt, \emph{Die {B}runn-{M}inkowskische {U}ngleichung und ihr
  {S}piegelbild sowie die isoperimetrische {E}igenschaft der {K}ugel in der
  euklidischen und nichteuklidischen {G}eometrie. {I}}, Math. Nachr. \textbf{1}
  (1948), 81--157. \MR{28600}

\bibitem{MR34044}
\bysame, \emph{Die {B}runn-{M}inkowskische {U}ngleichung und ihr {S}piegelbild
  sowie die isoperimetrische {E}igenschaft der {K}ugel in der euklidischen und
  nichteuklidischen {G}eometrie. {II}}, Math. Nachr. \textbf{2} (1949),
  171--244. \MR{34044}

\bibitem{MR2083401}
Carsten Sch\"{u}tt and Elisabeth Werner, \emph{Polytopes with vertices chosen
  randomly from the boundary of a convex body}, Geometric aspects of functional
  analysis, Lecture Notes in Math., vol. 1807, Springer, Berlin, 2003,
  pp.~241--422. \MR{2083401}

\bibitem{MR2446328}
Sasha Sodin, \emph{An isoperimetric inequality on the {$l_p$} balls}, Ann.
  Inst. Henri Poincar\'{e} Probab. Stat. \textbf{44} (2008), no.~2, 362--373.
  \MR{2446328}

\bibitem{MR517198}
V.~N. Sudakov, \emph{Typical distributions of linear functionals in
  finite-dimensional spaces of high dimension}, Dokl. Akad. Nauk SSSR
  \textbf{243} (1978), no.~6, 1402--1405. \MR{517198}

\bibitem{MR3802312}
Christoph Th\"{a}le, \emph{Central limit theorem for the volume of random
  polytopes with vertices on the boundary}, Discrete Comput. Geom. \textbf{59}
  (2018), no.~4, 990--1000. \MR{3802312}

\bibitem{MR1964483}
C\'{e}dric Villani, \emph{Topics in optimal transportation}, Graduate Studies
  in Mathematics, vol.~58, American Mathematical Society, Providence, RI, 2003.
  \MR{1964483}

\bibitem{MR1092403}
Richard~A. Vitale, \emph{Symmetric statistics and random shape}, Proceedings of
  the 1st {W}orld {C}ongress of the {B}ernoulli {S}ociety, {V}ol. 1
  ({T}ashkent, 1986), VNU Sci. Press, Utrecht, 1987, pp.~595--600. \MR{1092403}

\bibitem{MR2221249}
V.~H. Vu, \emph{Sharp concentration of random polytopes}, Geom. Funct. Anal.
  \textbf{15} (2005), no.~6, 1284--1318. \MR{2221249}

\end{thebibliography}
\end{document}